\documentclass[preview,hidelinks,onefignum,onetabnum]{siamart251216}

\usepackage{amsgen,amssymb,amsfonts,amsbsy,amsmath,amscd}
\usepackage{comment} 
\usepackage{tikz}
\usepackage{pgfplots}
\usetikzlibrary{external}
\usepgfplotslibrary{groupplots}
\usepackage{caption}
\usepackage{lipsum}
\usepackage{graphicx}
\usepackage{epstopdf}
\usepackage{algorithmic}
\ifpdf
  \DeclareGraphicsExtensions{.eps,.pdf,.png,.jpg}
\else
  \DeclareGraphicsExtensions{.eps}
\fi

\renewcommand{\SS}{\mathcal{S}}
\newcommand{\KK}{\mathcal{K}}

\newsiamremark{remark}{Remark}
\newsiamremark{hypothesis}{Hypothesis}
\crefname{hypothesis}{Hypothesis}{Hypotheses}
\newsiamthm{claim}{Claim}
\newsiamremark{fact}{Fact}
\crefname{fact}{Fact}{Facts}

\headers{Spectral theory of nonlocal plasmonics}{H. Lee, M. Ruiz, and S. Yu}

\title{Spectral theory of plasmonic resonances in the nonlocal regime\thanks{
\funding{	H.L. was supported by the \textit{National Research Foundation of Korea} (NRF) grant No. NRF-2018R1D1A1B07042678 and grant No. NRF-2022R1A4A5033271. M.R. was supported by \textit{The Royal Society} grant No. IES-R1-221233. S.Y. was supported by the \textit{National Research Foundation of Korea} (NRF) grant No. RS-2025-24523302 and from \textit{Korea University} grant No. K1923451}}
}

\author{Hyundae Lee\thanks{Department of Mathematics, Inha University, Incheon, Korea 
  (\email{hdlee@inha.ac.kr}).}
\and Matias Ruiz\thanks{{School of Computing and Mathematical Sciences, University of Leicester, Leicester, UK
  (\email{mr447@leicester.ac.uk}).}}
\and Sanghyeon Yu\thanks{Department of Mathematics, Korea University, Seoul, Korea 
  (\email{sanghyeon\_yu@korea.ac.kr}).}
}

\begin{document}

\maketitle

\begin{abstract}
We present a rigorous spectral analysis of plasmonic resonances in the nonlocal regime of spatially dispersive media. We adopt the quasi-static approximation of the hydrodynamic Drude model, which provides an analytically tractable setting to account for nonlocal effects. By reformulating the governing equations as a boundary integral system, we obtain an analytic Fredholm operator pencil that characterizes resonant behaviour. This framework enables the study of nonlocal plasmonic eigenvalues in general bounded Lipschitz domains, together with a corresponding resonant expansion of the scattered field. Our main result reveals a fundamental change in spectral structure: in contrast to the local theory—which exhibits infinitely many surface plasmon modes and field singularities in domains with corners—the nonlocal model admits a discrete real spectrum with a single accumulation point at a strictly positive value. Consequently, only finitely many surface plasmon modes exist, and the singular behaviour associated with sharp geometries is regularized. As such, our results show that nonlocality does not merely shift plasmonic resonances but completely reshapes the spectral landscape in a way that provides a mathematically transparent explanation for several nonlocal phenomena previously observed in numerical studies.
\end{abstract}

\begin{keywords}
	plasmonic resonances, nonlocal plasmonics, hydrodynamic Drude model, spectral theory, layer potentials.
\end{keywords}

\begin{MSCcodes}
35P05, 45B05, 31B10, 78A25
\end{MSCcodes}

	\section{Introduction}
	Plasmonic resonances in metallic nanoparticles correspond to coherent oscillations of the particle’s electron gas that, at specific excitation frequencies, generate resonant modes capable of strongly enhancing the near-field electric field \cite{Maier:07}. This phenomenon, first described in classical electrodynamics, has attracted renewed interest due to applications in sensing, imaging, nonlinear optics, and metamaterials.
	
	A common approach to their study is based on the quasi-static approximation of Maxwell’s equations, in which the electric field is represented by an electrostatic potential satisfying Gauss’s law, together with a local, piecewise-constant permittivity model \cite{Bohren:Book}. Within this setting, the problem admits a boundary integral formulation involving the Neumann-Poincaré operator, whose spectral properties provide a natural analytical framework for the characterization of plasmonic resonances \cite{Ando:16, Ammari:17, Ammari_book:18}.
	
	The quasi-static approximation is valid in the subwavelength regime, where the characteristic size of the particle is significantly smaller than the wavelength of the incident field. However, when the particle size decreases to only a few nanometres, the quantum-mechanical nature of the electron gas becomes significant, leading to the breakdown of the classical local constitutive relations \cite{Boroviks2022}. As such, the continued miniaturization of plasmonic and photonic architectures has made the inclusion of spatial-dispersion corrections a necessity, positioning nonlocality as one of the major frontiers in the study of optical materials and metamaterials \cite{NonlocalRoadmap25}. Moreover, even for moderately sized particles where the quasi-static approximation remains valid, the spectral properties of the local plasmonic spectral problem present features that are at odds with microscopic physics. Indeed, in smooth domains, the local theory predicts infinitely many boundary-localised eigenmodes \cite{Schnitzer:2019}, while in non-smooth geometries it gives rise to highly oscillatory corner modes associated with continuous spectra \cite{BonnetBenDhia:2013, Bonnetier:17, perfekt:2021}. These mathematical features correspond to surface potentials with arbitrarily rapid spatial oscillations—behaviour incompatible with microscopic considerations, where electron-electron repulsion and finite screening lengths prevent such extreme localisation.
	
	A convenient way to incorporate nonlocal effects, while
	avoiding the usual hurdles of a full quantum many-body Hamiltonian, is through semiclassical continuum models. Among these, the nonlocal hydrodynamic Drude model, in which the conduction electrons are treated as a charged compressible fluid, has become a primary framework for describing spatially dispersive effects in metallic nanostructures \cite{ciraci:2013, Raza:15}. Originally proposed in the context of electron-gas hydrodynamics \cite{Boardman:1982}, the model has since been extensively employed in nanoplasmonics to capture qualitative features of nonlocal response, including resonance blue shifts \cite{garcia:2008, yang:2020}, field screening at sub-nanometric scales \cite{wiener:2012, yang:2020}, and the saturation of near-field enhancement \cite{Fernandez:12, Schnitzer:16b, Schnitzer:16}. Recent works have also investigated analytical and numerical aspects of the model, including well-posedness \cite{Mystilidis:24}, multiscale analysis \cite{Schnitzer:16b}, numerical computation of resonances \cite{Binkowski:2019}, as well as solutions based on finite and boundary elements \cite{Hiremath:12, SCHMITT:2016, Yan:13}.
	
	Many of these studies highlight the importance of understanding the problem from a modal or spectral perspective, since this provides a natural language for interpreting observable quantities such as optical cross-sections and near-field intensities. While in the classical quasi-static theory this modal structure arises naturally from the spectral analysis of the Neumann–Poincaré operator, a comparable mathematically-rigorous spectral theory in the nonlocal regime has not yet been fully established. This work addresses this gap by establishing a spectral framework for nonlocal plasmonics. In particular, we introduce a precise operator-theoretic definition of nonlocal plasmonic eigenvalues and eigenmodes, and study their qualitative properties in general Lipschitz domains.

	We adopt the linearised quasi-static reduction of the hydrodynamic model, which consists of a Laplace equation for the electrostatic potential, coupled with a Helmholtz-type equation for the longitudinal charge-density fluctuation, and subject to a set of transmission conditions at the material interface. In direct analogy with the local quasi-static model, we use layer potential techniques to recast this system as a single permittivity-dependent boundary-integral equation that fully encapsulates the physical coupling of the model. Within this formulation, the nonlocal plasmonic spectrum is characterised as a set of poles of the associated operator-valued map, which are analysed using analytic Fredholm theory.
	
	A key outcome of our analysis is that the associated operator remains Fredholm for any bounded Lipschitz domain. Consequently, the corner–induced spectral pathologies of the local quasi-static model—arising from the loss of compactness of the Neumann–Poincaré operator—are suppressed, and well-defined plasmonic modes persist, even in the presence of geometric singularities. We further show that the nonlocal plasmonic eigenvalue problem admits only finitely many surface modes. This, again, stands in sharp contrast with the local quasi-static theory, which predicts infinite sequences of surface modes accumulating at the surface-plasmon limit \cite{Schnitzer:2019}; to the best of our knowledge, such a finiteness result has not been discussed previously, even at the numerical level. In addition, the analytic Fredholm structure of the operator family yields a pole expansion of the resolvent, providing a rigorous modal representation of solutions to the driven nonlocal system. 
	
	Taken together, these results show that nonlocality fundamentally alters the structure of the plasmonic spectral problem, introducing properties that, through the associated modal expansion, lead to clear and observable physical consequences. In particular, the direct correspondence between resonant electromagnetic fields and nonlocal plasmonic modes established by the modal expansion, together with the finiteness of surface modes in any Lipschitz domain, provides a natural framework for interpreting phenomena such as the regularisation of corner singularities \cite{wiener:2012, ciraci:2013} and the emergence of multi-peak excitation patterns under near-field excitation \cite{Christensen:2014}. 
	
	While this work focuses on the hydrodynamic model, our approach suggests a general framework which may characterise how higher-order corrections act as a spectral regularisation mechanism, restoring Fredholmness and reorganising the spectrum. Indeed, the hydrodynamic model considered here may be viewed as a fourth-order singular perturbation of an underlying second-order elliptic problem. Structures of this type arise in a range of physical settings, including generalisations of the hydrodynamic Drude model \cite{Mortensen:2014}, quasi-incompressible Cahn--Hilliard fluids \cite{Lowengrub:1998}, and gradient-type theories in nonlocal elasticity \cite{Eringen:1983}; in all of these settings, the corresponding unperturbed models fail to be Fredholm in the presence of material and geometric singularities \cite{mitrea:2002}. In such contexts, the present approach is particularly well suited: the classical jump relations of layer potentials provide a flexible mechanism to incorporate the additional boundary conditions induced by higher-order terms, leading to an operator-theoretic framework amenable to analytic Fredholm theory. In this way, the present work opens a door to the rigorous analysis of a broader class of higher-order singular perturbations of second-order elliptic systems.
	
	This paper is organised as follows. In Section~\ref{sec: nonlocal model} we introduce the linearised quasi-static hydrodynamic Drude model that forms the basis of our study. In Section~\ref{sec: layer potential formulation}, we first recall the fundamentals of layer potentials, which serve as the main analytical tool throughout the paper, and then reformulate the nonlocal problem as a system of boundary-integral equations. This formulation is then used in Section~\ref{sec: nonlocal plasmonic eigenvalue} to define nonlocal plasmonic eigenvalues and eigenmodes, and to establish our main result on the qualitative structure of the spectrum in general Lipschitz domains. We also discuss the implications of this result and show how the framework recovers previously derived eigenvalue perturbation formulas. Finally, in Section~\ref{sec: modal expansion} we construct a modal expansion and use it to analyse different phenomena previously observed in numerical studies. The paper concludes in Section~\ref{sec: conclusion} with final remarks and directions for future work.
	
	Throughout this paper, $H^s(X)$ denotes the usual Sobolev space of order $s$ in a space $X$, and $\langle\cdot,\cdot\rangle_{L^2(X)}$ denotes the standard $L^2(X)$ inner-product.
	
	\section{The nonlocal quasi-static hydrodynamic Drude model}\label{sec: nonlocal model}
	Consider a bounded Lipschitz domain \( D \subset \mathbb{R}^3 \) representing a metallic nanoparticle embedded in vacuum. To account for spatial dispersion in the light–matter interaction, we adopt the nonlocal hydrodynamic Drude model, which describes the conduction electrons in the metal as a charged compressible fluid governed by hydrodynamic equations.
	
	We work in the quasi-static approximation, which neglects radiative losses and is appropriate when the particle size is much smaller than the wavelength of the incident radiation. In this regime, the electric field derives from a microscopic electrostatic potential \(\phi\). Therefore, Gauss’ law inside the metal becomes
	\[
	-\Delta \phi = \frac{\tilde{\rho}}{\varepsilon_0}\qquad \text{in } D,
	\]
	where $\tilde{\rho}$ is the induced charge density, which after linearisation of the hydrodynamic equations \cite{Schnitzer:16b, Raza:15}, satisfies the relation
		\[
	\frac{\beta^{2}}{\omega_p^2}\Delta \tilde{\rho}
	-
	\frac{\varepsilon}{\varepsilon-1}\tilde{\rho}
	= 0
	\qquad \text{in } D.
	\]
	The above is derived assuming the Drude permittivity model \cite{Maier:07}
	\begin{equation}\label{eq:Drude}
		\varepsilon(\omega)
		= 1 - \frac{\omega_p^{2}}{\omega^{2} + i\gamma\omega},
	\end{equation}
where $\omega_p$ denotes the plasma frequency and $\gamma$ the damping rate.
	
In the exterior region \(\mathbb{R}^3\setminus\overline{D}\) the medium is vacuum, so there are no conduction electrons and Gauss’ law gives
\[
\Delta \phi = \tilde{f}
\qquad \text{in } \mathbb{R}^3\setminus\overline{D},
\]
where \(\tilde{f}\) has compact support outside $D$ and represents a near-field excitation source. We also impose the decay condition
\[
|\tilde{v}(x)-\tilde{g}(x)|\to0
\qquad \text{as } |x|\to\infty,
\]
where \(\tilde{g}\) is a harmonic function representing a far-field excitation source. 
	
Let \(\nu\) denote the outward unit normal to \(\partial D\). Since \(\phi\) is a microscopic electric potential both inside and outside the particle, the standard transmission conditions for the electric field on \(\partial D\) become
\[
\left.\phi\right|_-=\left.\phi\right|_+,
\qquad
\left.\frac{\partial \phi}{\partial \nu}\right|_-
=
\left.\frac{\partial \phi}{\partial \nu}\right|_+
=
-h\frac{\partial \tilde{\rho}}{\partial \nu}
\qquad \text{on } \partial D.
\]
The last identity follows from the hard-wall condition of the hydrodynamic model, which prevents electron spill-out by requiring that the normal component of the induced current vanishes on the boundary. Here and throughout, the subscripts $+$ and $-$ denote limits taken from the exterior and interior of $D$, respectively.

	\paragraph{Final system}
	Following \cite{Schnitzer:16b}, we introduce a dimensionless formulation of the problem. Let \(L\) denote a characteristic length of the particle and define
	\[
	h := \frac{\beta}{L\omega_p},
	\]
	which measures the strength of nonlocal effects and represents a dimensionless screening length. Define the rescaled variables
	\[
	\rho := \frac{\beta}{\omega_p\varepsilon_0}\tilde{\rho},
	\qquad
	u := \frac{\phi|_{D}}{L},
	\qquad
	v := \frac{\phi|_{\mathbb{R}^3 \setminus \overline{D}}}{L}.
	\]
	In these variables, the problem corresponds to finding
	\[
	\rho \in H^{1}(D),\quad
	u \in H^{1}(D),\quad
	v \in H^{1}_{\mathrm{loc}}(\mathbb{R}^3 \setminus \overline{D})
	\]
	such that the following system holds,
	\begin{equation}\label{eq:nonlocal scattering problem}
		\left|
		\begin{array}{ll}
			
			h^{2}\Delta \rho
			= \dfrac{\varepsilon}{\varepsilon - 1}\rho & \text{in } D,\\[4pt]
			
			h\Delta u = -\rho & \text{in } D,\\[4pt]
			
			\Delta v = f & \text{in } \mathbb{R}^3 \setminus \overline{D},\\[6pt]
			
			u = v & \text{on } \partial D,\\[4pt]
			
			\dfrac{\partial u}{\partial\nu}
			=
			\dfrac{\partial v}{\partial\nu}
			=
			-h\dfrac{\partial \rho}{\partial\nu}
			& \text{on } \partial D,\\[6pt]
			
			|v(x) - g(x)| \to 0 & \text{as } |x| \to \infty .
			
		\end{array}
		\right.
	\end{equation}
	
	The well-posedness of this system in the above functional framework will be analysed in the following sections.
	
	\section{Boundary integral formulation of the nonlocal model}\label{sec: layer potential formulation}
	Before reformulating the nonlocal excitation problem \eqref{eq:nonlocal scattering problem} as a system of boundary-integral equations, we recall elements of layer potential theory to fix notation and prove an operator decomposition result that will underpin the subsequent spectral analysis.
	\subsection{Layer potentials and operators decomposition}\label{sec: layer potentials}
	Let $D\subset\mathbb{R}^3$ be a bounded open set with a Lipschitz boundary. For $k\in\mathbb{C}$, denote by $\SS^k$ the single layer potential defined by
	\[
	\SS^k[\psi](x)
	:=
	\int_{\partial D}\Gamma^k(x-y)\psi(y)\,d\sigma(y),
	\qquad x\in\mathbb{R}^3,
	\]
	where $\Gamma^k$ denotes the outgoing fundamental solution of the Helmholtz operator $\Delta+k^2$, given by
	\begin{equation}\label{eq: Green function}
		\Gamma^k(x)=-\frac{1}{4\pi |x|}e^{ik|x|}.
	\end{equation}
	
	It is well-known (see, for instance, \cite{Ammari_book:09}) that 
	\[
	\SS^k:H^{-\frac{1}{2}}(\partial D)\to H^{\frac{1}{2}}(\partial D)\quad(\text{or }H^1_{\mathrm{loc}}(\mathbb R^3))
	\]
	is bounded and that the associated normal derivatives satisfy the jump relation
	\begin{equation}\label{jump1}
		\left.\frac{\partial}{\partial\nu}\SS^k[\psi]\right|_{\pm}
		=
		\left(\pm\frac12 I+\KK^{*,k}\right)[\psi],
		\qquad x\in\partial D ,
	\end{equation}
	where $\KK^{*,k}$ denotes the adjoint double-layer operator defined by
	\[
	\KK^{*,k}[\psi](x)
	=
	\mathrm{p.v.}
	\int_{\partial D}
	\frac{\partial \Gamma^k(x-y)}{\partial\nu(x)}
	\psi(y)\,d\sigma(y),
	\qquad x\in\partial D.
	\]
	Here $\mathrm{p.v.}$ stands for \textit{principal value}, $\partial/\partial \nu$ denotes differentiation in the outward normal direction. The operator
	\[
	\KK^{*,k}:H^{-\frac{1}{2}}(\partial D)\to H^{-\frac{1}{2}}(\partial D)
	\]
	is bounded for any bounded Lipschitz domain $D$ \cite{Ammari_book:09}.
	
	We write
	\[
	\SS := \SS^0, \qquad \KK^* := \KK^{*,0},
	\]
	which correspond to the static single-layer and  Neumann–Poincaré operator, respectively.
	\medskip
	
	\noindent\textbf{Decomposition of the operators.}
	We now introduce the operators
	\begin{align}
		\SS_1^k[\psi](x)
		&:=
		\int_{\partial D}
		\Gamma_1^k(x-y)\psi(y)\,d\sigma(y),\label{eq: S_1}\\
		\KK_1^{*,k}[\psi](x)
		&:=
		\int_{\partial D}
		\frac{\partial \Gamma_1^k(x-y)}{\partial\nu(x)}
		\psi(y)\,d\sigma(y), \label{eq: K_1}
	\end{align}
	where $\Gamma_1^k$ is defined through the decomposition
	\begin{equation}\label{eq: decomposition Gamma_k}
		\Gamma^k(x)=\Gamma^0(x)-\frac{ik}{4\pi}+k^2\Gamma_1^k(x).
	\end{equation}

    We have the following result.
	\begin{lemma}\label{lem: S_D1 and K_D1}
		Let $D$ be a bounded Lipschitz domain in $\mathbb{R}^3$. The operators
		\[
		\SS_1^k:H^{-\frac12}(\partial D)\to H^{\frac12}(\partial D)
		\qquad
		\KK_1^{*,k}:H^{-\frac12}(\partial D)\to H^{-\frac12}(\partial D)
		\]
		are compact for all $k\in\mathbb{C}$. Moreover, the operator-valued maps
		$k\mapsto \SS_1^k$ and $k\mapsto \KK_1^{*,k}$ are analytic in $\mathbb{C}$.
	\end{lemma}
	
	\begin{proof}
		We first prove the compactness of $\SS_1^k:H^{-\frac12}(\partial D)\to H^{\frac12}(\partial D)$. Consider this operator as acting from $H^{-\frac12}(\partial D)$ to $H^1_{\mathrm{loc}}(\mathbb R^3)$, which is bounded for any bounded Lipschitz domain $D$. Since
		\[
		(\Delta+k^2)\Gamma^k=-\delta,
		\qquad
		\Delta\Gamma^0=-\delta,
		\]
		it follows, in the sense of distributions in $\mathbb R^3$, that
		$
		\Delta\Gamma_1^k=-k^2\Gamma^k
		$, and consequently,
		$
		\Delta S_1^k[\psi]=-k^2 S^k[\psi].
		$
		Since $S_1^k[\psi]\in H^1_{\mathrm{loc}}(\mathbb R^3)$ and $S^k[\psi]\in
		H^1_{\mathrm{loc}}(\mathbb R^3)$, we have that
		$
		\Delta S_1^k[\psi]\in H^1_{\mathrm{loc}}(\mathbb R^3),
		$
		and therefore
		$
		S_1^k[\psi]\in H^3_{\mathrm{loc}}(\mathbb R^3)
		$, which in turn implies, after taking traces on $\partial D$ that
		$
		S_1^k[\psi]\in H^{5/2}(\partial D).
		$
		
		Hence, the operator
		$
		S_1^k:H^{-\frac12}(\partial D)\to H^{5/2}(\partial D)
		$
		is bounded, and since the embedding
		$
		H^{5/2}(\partial D)\hookrightarrow H^{1/2}(\partial D)
		$
		is compact on the compact Lipschitz surface $\partial D$, it follows that
		$
		S_1^k:H^{-\frac12}(\partial D)\to H^{1/2}(\partial D)
		$
		is compact.
		
		\medskip
		
		We next prove the compactness of $\KK^{*,k}:H^{-1/2}(\partial D)\to H^{-1/2}(\partial D)$. From
		\eqref{eq: decomposition Gamma_k} and a Taylor expansion of $\Gamma^0$
		near the origin we obtain
		\begin{equation}\label{eq: S_1^k decomposition}
			\Gamma_1^k(x)=c|x|+k|x|^2\upsilon(k|x|),
		\end{equation}
		where $c$ is a constant and $\upsilon$ is analytic. It follows that
		\[
		\kappa_x:=\frac{\partial\Gamma_1^k(x-\cdot)}{\partial\nu(x)}
		\in H^1(D).
		\]
		Let $\tau$ denote the trace operator on $\partial D$. Then
		$\tau(\kappa_x)\in H^{1/2}(\partial D)$, and therefore by Hölder's
		inequality
		\[
		\|\KK^{*,k}[\psi]\|_{L^2(\partial D)}
		\le
		\|\tau(\kappa_x)\|_{H^{1/2}(\partial D)}
		\|\psi\|_{H^{-1/2}(\partial D)} .
		\]
		Integrating over $\partial D$ yields
		\[
		\|\KK^{*,k}[\psi]\|_{L^2(\partial D)}
		\le C\|\psi\|_{H^{-1/2}(\partial D)},
		\]
		for some constant $C$ independent of $x$. 
		
		Hence
		$
		\KK^{*,k}:H^{-1/2}(\partial D)\to L^2(\partial D)
		$
		is bounded, and since the embedding
		$
		L^2(\partial D)\hookrightarrow H^{-1/2}(\partial D)
		$
		is compact on the Lipschitz surface $\partial D$, it follows that
		$
		\KK^{*,k}:H^{-1/2}(\partial D)\to H^{-1/2}(\partial D)
		$
		is compact.
		
		\medskip
		
		Finally, \eqref{eq: S_1^k decomposition} shows that the kernels
		\[
		\Gamma_1^k(x-y),
		\qquad
		\frac{\partial\Gamma_1^k(x-y)}{\partial\nu(x)}
		\]
		depend analytically on $k$. This directly implies that the
		operator-valued maps $k\mapsto \SS_1^k$ and
		$k\mapsto \KK_1^{*,k}$ are analytic.
	\end{proof}
	
	\subsection{Re-formulation of the nonlocal system}
	Let
	\begin{equation}\label{eq: definition of z^2}
		z:=\sqrt{\frac{\varepsilon}{1-\varepsilon}},
	\end{equation}
	and consider three boundary densities $\zeta,\,p,\,q\in H^{-\frac{1}{2}}(\partial D)$. We seek a solution of \eqref{eq:nonlocal scattering problem} in the form
	\begin{equation}\label{eq: layer potential representation}
		\rho= \SS^{\frac{z}{h}}[\zeta], \qquad 
		u= \frac{h}{z^2} \rho + \SS[p],\qquad 
		v = \SS[q] + v^{\text{ext}},
	\end{equation}
	where $\SS^{\frac{z}{h}}$ and $\SS$ denote single layer potentials defined in Section~\ref{sec: layer potentials}, and 
	\begin{equation}\label{eq: v^ext}
		v^{\mathrm{ext}}(x)
		:=
		\int_{\mathbb{R}^3}\Gamma^0(x,y)f(y)\,dy
		+
		g(x).
	\end{equation}
We obtain the following result.
	\begin{proposition} \label{prop1}
		For a given $\varepsilon\in\mathbb{C}$, $\varepsilon\neq0$, let $z\in\mathbb{C}$ satisfy \eqref{eq: definition of z^2}, and let $\psi\in H^{-\frac{1}{2}}(\partial D)$ solve
		\begin{align}\label{eq: nonlocal integral equation}
			\Lambda_h(z)[\psi]=  -z^2 \frac{\partial v^{\text{ext}}}{\partial \nu},
		\end{align}
		where 
		\[
		\Lambda_h(z):=\left( z^2 I+\frac12I +\KK^*\right)\left( -\frac12 I +\KK^{*,\frac{z}{h}} \right)-    \left( -\frac14 I +(\KK^*)^2 \right) (\SS)^{-1}  \SS^{\frac{z}{h}}.
		\]
		Then the boundary integral representation \eqref{eq: layer potential representation} solves the system \eqref{eq:nonlocal scattering problem}, with $\zeta = \psi/h$, and $p$ and $q$ defined by \eqref{eq: p(psi)} and \eqref{eq: q(psi)} as functions of $\zeta$ and $v^{\text{ext}}$.
		
		Conversely, if $\rho\in H^1(D)$, $u\in H^1(D)$, and $v\in H^1_{\mathrm{loc}}(\mathbb{R}^3\setminus\overline{D})$ solve \eqref{eq:nonlocal scattering problem}, and if $z$ is such that $\frac{z^2}{h^2}$ is not a Dirichlet eigenvalue of $D$ and $\frac{z}{h}$ is not a scattering resonance of $\SS^{\frac{z}{h}}$, then there exist $\zeta,p,q\in H^{-\frac{1}{2}}(\partial D)$ such that \eqref{eq: layer potential representation} holds.
	\end{proposition}
	
	\begin{proof}
		The representation \eqref{eq: layer potential representation} automatically satisfies the first three equations of \eqref{eq:nonlocal scattering problem}, namely the Helmholtz equation for $\rho$ and the Laplace equations for $u$ and $v$. To determine $\zeta$, $p$, and $q$, we impose the boundary conditions. Using the jump relations \eqref{jump1}, we obtain a coupled system for the boundary densities:
		\begin{align}
			p &= q -\frac{h}{z^2}\SS^{-1}[\rho] + \SS^{-1}[v^{\text{ext}}],\label{eq: p(psi)}\\
			\frac{z^2}{1+z^2} \left( -\frac12 I +\KK^* \right) p &= -h \frac{\partial \rho}{\partial \nu}, \nonumber\\
			\left( \frac12 I +\KK^* \right) q &= -h  \frac{\partial \rho}{\partial \nu} - \frac{\partial v^{\text{ext}}}{\partial \nu}. \label{eq: q(psi)}
		\end{align}
		Eliminating $p$ and $q$, we arrive at a single equation for $\rho$:
		\begin{align}\label{eq: equation rho}
			&\left(z^2 I+\frac12I +\KK^*\right) \left[\frac{\partial \rho}{\partial \nu}\right]  -  \left( -\frac14 I +(\KK^*)^2 \right) (\SS)^{-1}  [\rho] \nonumber\\
			&= -\frac{z^2}{h} \left( \left( -\frac14 I +(\KK^*)^2 \right) (\SS)^{-1}[v^{\text{ext}}] - \left( -\frac12 I +\KK^* \right) \left[\frac{\partial v^{\text{ext}}}{\partial \nu} \right] \right).
		\end{align}	
		The left-hand side of \eqref{eq: equation rho} can be further developed by using the representation $\rho = \SS^{\frac{z}{h}}[\zeta]$:
		\begin{align*}
			&\left( z^2 I+\frac12I +\KK^*\right) \left[\frac{\partial \rho}{\partial \nu}\right]  -  \left( -\frac14 I +(\KK^*)^2 \right) (\SS)^{-1}  [\rho] \nonumber\\
			&= \left(z^2 I+\frac12I +\KK^*\right)\left( -\frac12 I +\KK^{*,\frac{z}{h}} \right)[\zeta]  -   \left( -\frac14 I +(\KK^*)^2 \right) (\SS)^{-1}  \SS^{\frac{z}{h}}[\zeta].
		\end{align*}	
		The right-hand side of \eqref{eq: equation rho} can also be simplified by noting that $v^{\text{ext}}$ is harmonic in $D$, and therefore
		\[
		\frac{\partial v^{\text{ext}}}{\partial \nu} 
		=
		\left(-\frac12 I+\KK^*\right)\SS^{-1}[v^{\text{ext}}],
		\]
		which implies
		\[
		\left( -\frac14 I +(\KK^*)^2 \right) \SS^{-1}[v^{\text{ext}}]
		-
		\left( -\frac12 I +\KK^* \right)\!\left[\frac{\partial v^{\text{ext}}}{\partial \nu} \right]
		=
		\frac{\partial v^{\text{ext}}}{\partial \nu}.
		\]	
		Combining the previous identities proves the first claim of the proposition.
		
		\medskip
		
		To prove the converse, let $\rho$, $u$, and $v$ solve \eqref{eq:nonlocal scattering problem}. Since $\SS$ is invertible, and $\SS^{\frac{z}{h}}$ is invertible whenever $\frac{z}{h}$ is neither a Dirichlet eigenvalue of $D$ nor a scattering resonance of $\SS^{\frac{z}{h}}$, the representation \eqref{eq: layer potential representation} holds with
		\[
		\zeta := \left(\SS^{\frac{z}{h}}\right)^{-1}[\rho], \qquad 
		p := \SS^{-1}[u] + h \frac{\varepsilon-1}{\varepsilon}\SS^{-1}[\rho],\qquad  
		q := \SS^{-1}\left[v^{\text{ext}} - v\right].
		\]
	\end{proof}
	
	\medskip 
	The integral equation \eqref{eq: nonlocal integral equation} may be viewed as a nonlocal analogue of the local boundary integral formulation \eqref{eq: NP integral equation}, in the sense that it fully characterises the nonlocal quasi-static hydrodynamic Drude problem \eqref{eq:nonlocal scattering problem}. In fact, in smooth domains, the operator 
	\(\Lambda_h(z)\) may be interpreted as a singular perturbation (up to a multiplicative constant) of 
	\begin{equation*}
		z^2I+ \frac{1}{2}I + \KK^*
	\end{equation*}
	as \(h\to0\). This observation will be further discussed in 
	Section~\ref{sec: spectral asymptotics} in the context of eigenvalue perturbations.
	
	Note that, starting from \eqref{eq: equation rho}, one could alternatively formulate an integral equation directly for the trace of $\rho$, similar to \cite{Yan:13}:
	\begin{align}\label{eq: equation rho 2}
		\left( z^2 I+\frac12 I +\KK^*\right) \mathrm{DtN}_h(z)[\rho]
		-
		\left( -\frac14 I +(\KK^*)^2 \right) (\SS)^{-1}[\rho]
		=
		-\frac{z^2}{h}\frac{\partial v^{\mathrm{ext}}}{\partial \nu}.
	\end{align}
	Here $\mathrm{DtN}_h(z)$ denotes the Dirichlet-to-Neumann map associated with the interior Helmholtz problem
	\[
	\Delta \rho + \frac{z^2}{h^2}\rho = 0 \quad \text{in } D.
	\]
	
	While formally equivalent, this formulation involves the Helmholtz Dirichlet-to-Neumann operator and therefore inherits its singularities at the interior Dirichlet spectrum. In particular, the operator in \eqref{eq: equation rho 2} becomes unbounded at those values, which complicates the application of analytic Fredholm theory.
	
	The full boundary integral formulation \eqref{eq: nonlocal integral equation} is instead well defined for all $z\in\mathbb{C}$, but it introduces additional non-physical resonances arising from the radiation condition in the Helmholtz layer potentials. These correspond to scattering resonances of the single layer potential: there exist $\psi$ and complex $z$ such that
	\[
	\SS^{\frac{z}{h}}[\psi] = 0,
	\qquad
	\left( -\frac12 I +\KK^{*,\frac{z}{h}} \right)[\psi] = 0.
	\]
	These resonances are not physically relevant since they generate a vanishing field inside $D$. Moreover, due to the radiation condition satisfied by the fundamental solution \eqref{eq: Green function}, they lie in the lower half of the complex plane, which makes them easily distinguishable from the plasmonic resonances of interest.
	
	The framework of Proposition~\ref{prop1}, therefore, has the advantage of being amenable to analytic Fredholm theory, which will allow us to define and analyse the physically relevant nonlocal plasmonic resonances. To that end, it will be convenient to work with the following reformulation of \eqref{eq: nonlocal integral equation}.
	
	\begin{proposition} \label{prop2}
		The integral equation \eqref{eq: nonlocal integral equation} is equivalent to
		\begin{align}\label{eq: second nonlocal integral equation}
			\tilde{\Lambda}_h(z)[\psi] = \frac{\partial v^{\text{ext}}}{\partial \nu},
		\end{align}
		where
		\begin{align}
			\tilde{\Lambda}_h(z)&:= -\frac12 I +\KK^* + \mathcal{L}(z), \label{eq: tilde{Lambda}_h(z)}\\
			\mathcal{L}(z)&:=\frac{1}{h^2} \left( z^2 I+\frac12I +\KK^*\right) \KK_{1}^{*,\frac{z}{h}}  -\frac{1}{h^2}\left( -\frac14 I +(\KK^*)^2 \right) (\SS)^{-1} \SS_{1}^{\frac{z}{h}}.\nonumber
		\end{align}
		All the operators above are defined in Section~\ref{sec: layer potentials}. Moreover, the operator-valued map $z\mapsto\mathcal{L}(z)$ is analytic and $\mathcal{L}(z): H^{-\frac{1}{2}}(\partial D)\to H^{-\frac{1}{2}}(\partial D)$ is compact for all $z\in\mathbb{C}$.
	\end{proposition}
	
	\begin{proof}
		From the definitions of $\SS_{1}^{\frac{z}{h}}$ and $\KK_{1}^{*,\frac{z}{h}}$ in \eqref{eq: S_1} and \eqref{eq: K_1}, respectively, we have
		\begin{align*}
			\SS^{\frac{z}{h}}[\psi] &= \SS[\psi] - \frac{iz}{4\pi h} \int_{\partial D} \psi\, d\sigma + \frac{z^2}{h^2} \SS_{1}^{\frac{z}{h}}[\psi], 
			\\
			\KK^{*,\frac{z}{h}}[\psi] &= \KK^*[\psi] + \frac{z^2}{h^2}\KK_{1}^{*,\frac{z}{h}}[\psi]. 
		\end{align*}
		It is known that the value $1/2$ is always an eigenvalue of $\KK^*$ with eigenfunction proportional to $\SS^{-1}[\chi]$, where $\chi$ denotes the constant function on $\partial D$. Therefore, we have
		\[
		\left( -\frac14 I +(\KK^*)^2 \right) (\SS)^{-1}[\chi] = 0.
		\]
		Using the identities above, we obtain
		\begin{align*}
			\Lambda_h(z)
			&= \left( z^2 I+\frac12I +\KK^*\right)\left( -\frac12 I +\KK^{*,\frac{z}{h}} \right)-   \left( -\frac14 I +(\KK^*)^2 \right) (\SS)^{-1}  \SS^{\frac{z}{h}}\\
			&= z^2\left( -\frac12 I +\KK^*+ \underbrace{\frac{1}{h^2} \left( z^2 I+\frac12I +\KK^*\right) \KK_{1}^{*,\frac{z}{h}}  -\frac{1}{h^2}\left( -\frac14 I +(\KK^*)^2 \right) (\SS)^{-1} \SS_{1}^{\frac{z}{h}}}_{:= \mathcal{L}(z)}\right).
		\end{align*}
		Substituting this identity into \eqref{eq: nonlocal integral equation} yields \eqref{eq: second nonlocal integral equation}. The analyticity and compactness of $\mathcal{L}(z)$ follow directly from Lemma~\ref{lem: S_D1 and K_D1}.
	\end{proof}

	\section{The nonlocal plasmonic spectrum}\label{sec: nonlocal plasmonic eigenvalue}
	Before analysing the nonlocal problem, we briefly recall the spectral structure of the local quasi-static plasmonic model, which will serve as a point of comparison for our main results.
	\subsection{Review: the local plasmonic spectrum}
	In the local quasi-static regime, the electromagnetic response is described in terms of interior and exterior harmonic potentials \(u \in H^1(D)\) and \(v \in H^1_{\mathrm{loc}}(\mathbb{R}^3 \setminus \overline{D})\). These fields (solutions to the forced version of \eqref{eq:local plasmonic eigenvalue problem}) admit a representation in terms of static single-layer potentials with density \(\varphi \in H^{-1/2}(\partial D)\) (see, for instance, \cite{Ando:16, Ammari:17}), which satisfies the boundary-integral equation
	\begin{equation}\label{eq: NP integral equation}
		\left(
		z^2 I + \frac{1}{2}I + \KK^*
		\right)[\varphi]
		=
		\frac{\partial v^{\mathrm{ext}}}{\partial \nu},
	\end{equation}
	where \(v^{\mathrm{ext}}\) denotes the external forcing defined in \eqref{eq: v^ext}.
	
	This boundary integral formulation completely characterises the plasmonic response of the particle. As such, the spectral properties of the Neumann--Poincar\'e operator \(\KK^*\) determine the resonant behaviour: plasmonic resonances occur, through the relation \(z^2 + \frac{1}{2} = -\frac{\varepsilon + 1}{2(\varepsilon - 1)}\), at those values of the permittivity \(\varepsilon\) for which the operator fails to be invertible. Hence,  the local plasmonic resonant set is given by
	\begin{equation}\label{eq: sigma_local}
		\sigma_{\mathrm{local}}
		:=
		\left\{
		\varepsilon \in \mathbb{C} \;\middle|\;
		-\frac{\varepsilon + 1}{2(\varepsilon - 1)}I + \KK^*
		\text{ is not invertible}
		\right\}.
	\end{equation}
	We refer to \(\sigma_{\mathrm{local}}\) as the \emph{local plasmonic spectrum}.
	
	Equivalently, \emph{local plasmonic modes} and associated \emph{local plasmonic eigenvalues}, are characterised by the so-called \textit{plasmonic eigenvalue problem} \cite{Grieser:14}:
	
	\medskip
	\begin{center}
		\(\text{Find } \varepsilon \in \mathbb{C},\, u \in H^{1}(D), \text{ and } v \in H^{1}_{\mathrm{loc}}(\mathbb{R}^3 \setminus \overline{D}) \text{ such that}\)
		\begin{equation}\label{eq:local plasmonic eigenvalue problem}
			\left|
			\begin{array}{ll}
				\Delta u = 0, & \text{in } D,\\[3pt]
				\Delta v = 0, & \text{in } \mathbb{R}^3 \setminus \overline{D},\\[5pt]
				u = v, & \text{on } \partial D,\\[3pt]
				\varepsilon \frac{\partial u}{\partial \nu}
				= \frac{\partial v}{\partial \nu},
				& \text{on } \partial D,\\[5pt]
				|v(x)| \to 0, & \text{as } |x| \to \infty.
			\end{array}
			\right.
		\end{equation}
	\end{center}
	
	The following result, which stems directly from the well-known spectral properties of $\KK^*$  \cite{Perfekt:14, Ammari_book:18, perfekt:2019}, summarises the classical spectral properties of the local quasi-static model.
	\begin{theorem}\label{thm: theorem local plasmonic spectrum}
		Let \(D\) be a bounded Lipschitz domain. Then the local plasmonic spectrum \(\sigma_{\mathrm{local}}\) satisfies:
		\begin{enumerate}
			\item If \(\partial D\) is of type $C^{1,\alpha}$, with $\alpha>0$, then \(\sigma_{\mathrm{local}}\) is discrete and consists of negative real numbers accumulating at \(-1\).
			\item If \(\partial D\) contains a geometric singularity (e.g., corners in two dimensions, or edges and vertices in three dimensions), then \(\sigma_{\mathrm{local}}\) contains a continuous component. The associated (generalised) plasmonic eigenmodes $u$ and $v$ fail to belong to \(H^1(D)\) and \(H^{1}_{\mathrm{loc}}(\mathbb{R}^3 \setminus \overline{D})\), respectively.
		\end{enumerate}
	\end{theorem}

	\subsection{Definition of the nonlocal spectrum}	
	Excluding the scattering resonances associated with $\SS^{\frac{z}{h}}$, the integral equation \eqref{eq: second nonlocal integral equation}, fully captures the physics of the nonlocal system \eqref{eq:nonlocal scattering problem}. We therefore use \eqref{eq: second nonlocal integral equation}, in analogy with \eqref{eq: sigma_local}, to define the \textit{nonlocal plasmonic spectrum} that characterise resonant states.
	
	\begin{definition}[Nonlocal plasmonic spectrum]\label{def: nonlocal resonant set}
		We define the nonlocal plasmonic spectrum as
		\begin{equation}\label{eq: sigma_nonlocal}
			\sigma_{\mathrm{nonlocal}}
			:=
			\left\{
			\varepsilon=\frac{z^2}{z^2+1}\in \mathbb{C} \;\middle|\;
			\begin{array}{l}
				\tilde{\Lambda}_h(z)\ \text{is not invertible,}\\[4pt]
				\SS^{\frac{z}{h}} : H^{-\frac{1}{2}}(\partial D) \to H^1(D)\ \text{is injective.}
			\end{array}
			\right\}
		\end{equation}
	\end{definition}
	
	Considering $\tilde{\Lambda}_h(z)$ instead of $\Lambda_h(z)$ in Definition~\ref{def: nonlocal resonant set} is particularly useful since, as we will show below, it allows us to prove that for any Lipschitz domain $D\subset\mathbb{R}^3$ the resonant set is discrete. This implies, via Proposition~\ref{prop1} and Proposition~\ref{prop2}, the existence of \textit{nonlocal plasmonic modes} and associated \textit{nonlocal plasmonic eigenvalues} which satisfy the following \textit{nonlocal plasmonic eigenvalue problem}:
	
	\medskip
	\begin{center}
		$\mbox{Find }\varepsilon \in \mathbb{C},\,\rho \in H^{1}(D),\,u \in H^{1}(D),\text{ and }v \in H^{1}_{\mathrm{loc}}(\mathbb{R}^3 \setminus \overline{D})\mbox{ such that}$
		\begin{equation}\label{eq:nonlocal plasmonic eigenvalue problem}
			\left|
			\begin{array}{ll}
				
				h^{2}\Delta \rho 
				= \dfrac{\varepsilon}{\varepsilon - 1}\rho & \text{in } D,\\[3pt]
				
				h\Delta u = -\rho & \text{in } D,\\[3pt]
				
				\Delta v = 0 & \text{in } \mathbb{R}^3 \setminus \overline{D},\\[5pt]
				
				u = v & \text{on }\partial D,\\[3pt]
				
				\dfrac{\partial u}{\partial\nu}
				=
				\dfrac{\partial v}{\partial\nu}
				=
				-h\dfrac{\partial \rho}{\partial\nu} & \text{on }\partial D,\\[5pt]
				
				|v(x)| \to 0 & \text{as } |x| \to \infty.
			\end{array}
			\right.
		\end{equation}
	\end{center}
	
	\subsection{Main result}\label{sec: main result}
	
	The following theorem constitutes the main result of this paper.	
	\begin{theorem}\label{thm: main result}
		Let \(D\) be a bounded Lipschitz domain. Then the nonlocal plasmonic spectrum \(\sigma_{\mathrm{nonlocal}}\) satisfies:
		\begin{enumerate}
			\item The spectrum $\sigma_{\mathrm{nonlocal}}$ remains discrete regardless of the Lipschitz regularity of $\partial D$. The nonlocal plasmonic eigenmodes $\rho$, $u$, and $v$ always belong to  \(H^1(D)\), \(H^1(D)\), and \(H^{1}_{\mathrm{loc}}(\mathbb{R}^3 \setminus \overline{D})\), respectively.
			\item If $\sigma_{\mathrm{nonlocal}}$ is infinite, then the only possible accumulation point is $1$.
		\end{enumerate}
	\end{theorem}
	
	\begin{proof}
		It is known that, in $H^{-\frac{1}{2}}(\partial D)$,
		\begin{equation}\label{eq: spectrum of K*}
			\sigma(\KK^*)\subset(-1/2,1/2],
		\end{equation}
		and $1/2$ is always an eigenvalue with a finite-dimensional eigenspace. It follows that for any Lipschitz domain $D$, the operator $-\frac12 I +\KK^*$ is Fredholm of index zero. This, together with the fact that $\mathcal{L}(z)$ is an analytic family of compact operators, implies that $\tilde{\Lambda}_h(z)$ is an analytic family of Fredholm operators of index \(0\). Thus, by the analytic Fredholm theorem, either		
		\begin{enumerate}
			\item \(\tilde{\Lambda}_h(z)\) is nowhere invertible on \(\mathbb{C}\), or
			\item if \(\tilde{\Lambda}_h(z_0)\) is invertible for some \(z_0\in\mathbb{C}\), then the set
			\[
			\{ z\in\mathbb{C} : \tilde{\Lambda}_h(z)\ \text{is not invertible} \}
			\]
			is discrete, with no accumulation points in \(\mathbb{C}\).
		\end{enumerate}
		
		We now show that the second alternative occurs. We proceed by contradiction. Suppose there exists \(z\) with \(z^2\notin\mathbb{R}\) such that \(\tilde{\Lambda}_h(z)\) is not invertible. Hence, since $\tilde{\Lambda}_h(z)$ is Fredholm, there must exists $\psi$ such that
		\begin{equation}\label{eq: non-invertibility cond1}
			\tilde{\Lambda}_h(z)[\psi] = 0.
		\end{equation}
		Recall from Proposition~\ref{prop2} that
		\[
		z^2 \tilde{\Lambda}_h(z)[\psi] = \Lambda_h(z)[\psi].
		\]
		Thus, assuming without loss of generality that \(z \neq 0\), condition \eqref{eq: non-invertibility cond1} reduces to
		\[
		\Lambda_h(z)[\psi] = 0.
		\]
		If, in addition, \(z/h\) is not a scattering resonance of \(\SS^{\frac{z}{h}}\) (which may occur when \(\Im z < 0\)), we have from \eqref{eq: equation rho 2} that 
		\begin{align*}
			\left( z^2 I+\frac12 I +\KK^* \right) \left[\frac{\partial \rho}{\partial \nu} \right] = \left( -\frac14 I +(\KK^*)^2 \right) (\SS)^{-1} [\rho].
		\end{align*}
		Under the assumptions above, the operator $ z^2 I+\frac12 I +\KK^*$ is invertible since $\sigma(\KK^*)$ is real in $H^{-\frac12}(\partial D)$. Thus, we have 
		\begin{align}	\label{rhoboundary2}
			\frac{\partial \rho}{\partial \nu}
			=
			\left( z^2 I+\frac12 I+\KK^* \right)^{-1}
			\left( -\frac14 I+(\KK^*)^2 \right)\SS^{-1}[\rho].
		\end{align}        
		Recall also that \(\rho\) satisfies
		\begin{equation*}
			\Delta \rho + \frac{z^2}{h^2}\rho = 0
			\qquad \text{in } D.
		\end{equation*}
		Multiplying by \(\bar{\rho}\), integrating over \(D\), and integrating by parts, we obtain
		\begin{equation}\label{eq: ibp rho proof main thm}
			\|\nabla \rho\|_{L^2(D)}^2
			-
			\frac{z^2}{h^2}\|\rho\|_{L^2(D)}^2
			=
			\left\langle \rho\,,\,\frac{\partial \rho}{\partial \nu}\right\rangle_{L^2(\partial D)},
		\end{equation}
		where $\left\langle \cdot,\cdot\right\rangle_{L^2(\partial D)}$ is seen (here and throughout) as the dual pairing between $H^{\frac12}(\partial D)$ and $H^{-\frac12}(\partial D)$. Then, using the boundary identity \eqref{rhoboundary2} it follows that
		\begin{equation}\label{eq: boundary rho first}
			\|\nabla \rho\|_{L^2(D)}^2
			-
			\frac{z^2}{h^2}\|\rho\|_{L^2(D)}^2
			=
			\left\langle
			\rho\,,\,
			\left( z^2 I+\frac12 I+\KK^* \right)^{-1}
			\left( -\frac14 I+(\KK^*)^2 \right)\SS^{-1}[\rho]
			\right\rangle_{L^2(\partial D)}.
		\end{equation}
		Define now
		\[
		\xi=(z^2 I+\tfrac12 I+\KK)^{-1}[\rho],
		\]
		so that
		\[
		\rho=(z^2 I+\tfrac12 I+\KK)[\xi].
		\]
		Using the symmetrisation identity (see, for instance, \cite{Ammari_book:09})
		\begin{equation}\label{eq: symmetrisation identity}
			\KK^*\SS^{-1}=\SS^{-1}\KK,
		\end{equation}
		we obtain
		\[
		\SS^{-1}[\rho]
		=
		(z^2 I+\tfrac12 I+\KK^*)\SS^{-1}[\xi],
		\]
		which, inserted into \eqref{eq: boundary rho first}, gives
		\begin{align*}
			\|\nabla \rho\|_{L^2(D)}^2
			-
			\frac{z^2}{h^2}\|\rho\|_{L^2(D)}^2
			&=
			\left\langle
			(z^2 I+\tfrac12 I+\KK)[\xi]
			\,,\,
			\left(-\tfrac14 I+(\KK^*)^2\right)\SS^{-1}[\xi]
			\right\rangle_{L^2(\partial D)} \\
			&=
			\overline{z^2}
			\left\langle
			\xi
			\,,\,
			\left(-\tfrac14 I+(\KK^*)^2\right)\SS^{-1}[\xi]
			\right\rangle_{L^2(\partial D)} \\
			&\quad+
			\left\langle
			\xi
			\,,\,
			\left(\tfrac12 I+\KK^*\right)
			\left(-\tfrac14 I+(\KK^*)^2\right)\SS^{-1}[\xi]
			\right\rangle_{L^2(\partial D)}.
		\end{align*}
		It is known that $\SS^{\frac12}\KK^*\SS^{-\frac12}$ is self-adjoint in $L^2(\partial D)$ (see, for instance, \cite{Ammari_book:09}). This implies that both 
		\[
		\left\langle
		\xi
		\,,\,
		\left(-\tfrac14 I+(\KK^*)^2\right)\SS^{-1}[\xi]
		\right\rangle_{L^2(\partial D)}
		\]
		and
		\[
		\left\langle
		\xi
		\,,\,
		\left(\tfrac12 I+\KK^*\right)
		\left(-\tfrac14 I+(\KK^*)^2\right)\SS^{-1}[\xi]
		\right\rangle_{L^2(\partial D)}
		\]
		are real. Taking imaginary parts we deduce that
		\[
		\frac{1}{h^2}\|\rho\|_{L^2(D)}^2
		=
		\left\langle
		\xi
		\,,\,
		\left(-\tfrac14 I+(\KK^*)^2\right)\SS^{-1}[\xi]
		\right\rangle_{L^2(\partial D)}.
		\]
		Substituting this back yields
		\begin{align}
			\|\nabla \rho\|_{L^2(D)}^2
			&=
			2\Re(z^2)
			\left\langle
			\xi
			\,,\,
			\left(-\tfrac14 I+(\KK^*)^2\right)\SS^{-1}[\xi]
			\right\rangle_{L^2(\partial D)}
			\nonumber\\
			&\quad+
			\left\langle
			\xi
			\,,\,
			\left(\tfrac12 I+\KK^*\right)
			\left(-\tfrac14 I+(\KK^*)^2\right)\SS^{-1}[\xi]
			\right\rangle_{L^2(\partial D)}.
			\label{eq:grad-rho-identity}
		\end{align}
		
		\medskip
		
		Now let \(w\) be the harmonic function in \(D\) with boundary value
		\[
		w=\rho
		\qquad\text{on }\partial D.
		\]
		Then
		\[
		\|\nabla w\|_{L^2(D)}^2
		=
		\left\langle
		\rho
		\,,\,
		\frac{\partial w}{\partial \nu}
		\right\rangle_{L^2(\partial D)}
		=
		\left\langle
		\rho
		\,,\,
		\left(-\tfrac12 I+\KK^*\right)\SS^{-1}[\rho]
		\right\rangle_{L^2(\partial D)},
		\]
		where the last identity follows from the representation formula for the interior Dirichlet problem,
		\[
		\frac{\partial w}{\partial \nu}
		=
		\left(-\tfrac12 I+\KK^*\right)\SS^{-1}[\rho].
		\]
		Using again
		\[
		\rho=(z^2 I+\tfrac12 I+\KK)[\xi],
		\]
		we obtain
		\begin{equation}\label{eq: norm grad w}
			\|\nabla w\|_{L^2(D)}^2
			=
			\left\langle
			\xi
			\,,\,
			\left(\overline{z^2}I+\tfrac12 I+\KK^*\right)
			\left(z^2 I+\tfrac12 I+\KK^*\right)
			\left(-\tfrac12 I+\KK^*\right)\SS^{-1}[\xi]
			\right\rangle_{L^2(\partial D)}.
		\end{equation}
		
		\medskip
		
		Combining \eqref{eq:grad-rho-identity} with \eqref{eq: norm grad w} gives
		\[
		\|\nabla \rho\|_{L^2(D)}^2
		-
		\|\nabla w\|_{L^2(D)}^2
		=
		-|z^2|^2
		\left\langle
		\xi
		\,,\,
		\left(-\tfrac12 I+\KK^*\right)\SS^{-1}[\xi]
		\right\rangle_{L^2(\partial D)}
		\le 0,
		\]
		where the last inequality stems from \eqref{eq: spectrum of K*} and the fact that $\SS^{\frac12}\KK^*\SS^{-\frac12}$ is self-adjoint in $L^2(\partial D)$. On the other hand, by the Dirichlet principle,
		\[
		\|\nabla \rho\|_{L^2(D)}^2
		>
		\min_{\substack{\tilde{w}\in H^1(D)\\ \tilde{w}=\rho \text{ on }\partial D}}
		\|\nabla \tilde{w}\|_{L^2(D)}^2
		=
		\|\nabla w\|_{L^2(D)}^2,
		\]
		which is a contradiction. Therefore, for any \(z\) such that \(z^2 \notin \mathbb{R}\) and \(\Im z > 0\) (so as to exclude scattering resonances), the operator \(\tilde{\Lambda}_h(z)\) is invertible. This shows that the second alternative holds and, through the relation
		\[
		z^2=\frac{\varepsilon}{1-\varepsilon},
		\]
		the resonant set can only correspond to real values of \(\varepsilon\).
		
		Since the non-invertibility points of \(\tilde{\Lambda}_h(z)\) form a discrete subset of \(\mathbb{C}\), any infinite sequence of resonances must diverge to infinity in the \(z\)-plane, implying that the corresponding permittivity values can only accumulate at \(\varepsilon=1\).
	\end{proof}
	
	\begin{corollary}[Well-posedness]
		A direct consequence of Theorem~\ref{thm: main result} is the well-posedness of system \eqref{eq:nonlocal scattering problem} for any $\varepsilon$ that is not in $\sigma_{\mathrm{nonlocal}}$. In physical systems, \eqref{eq:nonlocal scattering problem} is always well-posed since the permittivity $\varepsilon$ necessarily has a positive imaginary part to account for Ohmic losses, and consequently, $z^2$ lies in the upper half-plane.
	\end{corollary}
	
	\subsection{Spectral regularization}
	We next discuss the contrast between Theorem~\ref{thm: theorem local plasmonic spectrum} with Theorem~\ref{thm: main result}, showing how $\sigma_{\text{nonlocal}}$ is fundamentally altered relative to $\sigma_{\text{local}}$.
	
	\paragraph{Discrete spectrum for Lipschitz domains}
	A first important difference concerns the structure of the plasmonic spectrum in the case of Lipschitz non-smooth geometries, such as domains with corners, edges, or conical singularities.
	
	In such geometries, the Neumann-Poincaré operator is no longer compact, and its spectrum develops a continuous component \cite{Perfekt:14, perfekt:2019}. As a result, local plasmonic quasi-modes arise, which remain bounded in \(L^2(D)\) (for the interior part of the mode) but fail to belong to \(H^1(D)\), with their singular behaviour concentrating near geometric irregularities. For example, in a two-dimensional domain containing a corner, \(L^2(D)\)-bounded solutions of the local plasmonic eigenvalue problem admit an ansatz of the form
	\[
	u(r,\theta)=g(\theta)\,e^{\mathrm{i}c\log r},
	\]
	where \(c\) is a positive constant and \((r,\theta)\) denote polar coordinates centred at the tip of the corner \cite{BonnetBenDhia:2013,bonnetbendhia:2016}. This results in the highly oscillatory behaviour of the mode as $r\to 0$.
	
	In contrast, the nonlocal plasmonic spectrum is discrete for any Lipschitz domain, including geometries with corners, edges, or conical singularities. Consequently, only genuine nonlocal plasmonic eigenvalues arise, each associated with a plasmonic mode that remains bounded in \(H^1(D)\). In fact, in the case of a two-dimensional corner, it can be shown (see Appendix~\ref{sec: Appendix corner}
	) that there exist constants \(C_{\rho}\), \(C_u\), and \(C_v\) such that the corresponding nonlocal plasmonic modes satisfy
	\[
	\rho(r,\theta)\to C_\rho, \qquad
	u(r,\theta)\to C_u,\qquad
	v(r,\theta)\to C_v,
	\quad\text{as } r\to0.
	\]	
	
	\paragraph{Finite number of surface plasmon modes} 
	A second important difference concerns the qualitative
	structure of the plasmonic spectrum in smooth domains.
	
	If $\partial D$ is of class $C^{1,\alpha}$, with $\alpha>0$, then the Neumann–Poincaré operator $\KK^*$ is compact on $H^{-\frac{1}{2}}(\partial D)$, and therefore its spectrum consists of a sequence of eigenvalues accumulating at \(0\). Consequently, the local plasmonic eigenvalues form a sequence of negative values accumulating at $-1$. The associated plasmonic modes are harmonic and become increasingly localized near the particle's boundary as the mode number increases. In the high–mode limit, these modes approach the surface plasmon polaritons supported by a flat metal–vacuum interface \cite{Schnitzer:2019}.
	
	In contrast, the nonlocal plasmonic eigenvalues can only accumulate at \(1\), implying a finite number of negative eigenvalues, and therefore a finite number of the associated surface-type modes, analogous to the classical plasmonic modes---the remaining of the spectrum of positive eigenvalues correspond to a qualitatively different family of bulk-type. Figure~\ref{fig: eigenvalues} illustrates this phenomenon by showing the negative part of the local and nonlocal spectrum of a sphere, for several values of the nonlocal parameter \(h\). 
	
	\begin{figure}[h]
		\centering
		\begin{tikzpicture}	
			\begin{axis}[
				width=10cm,
				height=6.2cm,
				xmin=0, xmax=60,
				ymin=-2, ymax=0,
				xlabel={$j$},
				ylabel={$\varepsilon_j$},
				grid=both,
				ticklabel style={font=\small},
				label style={font=\small},
				legend style={font=\footnotesize, at={(1,-0.005)}, row sep=-1pt, anchor=south east, draw=none, fill=none},
				line cap=round,
				]
				
				\addplot[color={rgb,255:red,200; green,90; blue,0}, thick, mark=+, mark size=2pt] table[col sep=tab] {eps_nonlocal_1.dat};
				\addlegendentry{$h = 5\times10^{-4}$}
				\addplot[color={rgb,255:red,200; green,90; blue,0}, thick, mark=o, mark size=2pt] table[col sep=tab] {eps_nonlocal_2.dat};
				\addlegendentry{$h = 1\times10^{-2}$}
				\addplot[color={rgb,255:red,200; green,90; blue,0}, thick, mark=diamond, mark size=2pt] table[col sep=tab] {eps_nonlocal_3.dat};
				\addlegendentry{$h = 2\times10^{-2}$}	
				\addplot[color={rgb,255:red,200; green,90; blue,0}, thick, mark = square, mark size=2pt] table[col sep=tab] {eps_nonlocal_4.dat};
				\addlegendentry{$h = 5\times10^{-2}$}
				\addplot[semithick, mark=*, mark size=0.8pt] table[col sep=tab] {eps_local.dat};
				\addlegendentry{local}		
			\end{axis}	
		\end{tikzpicture}
		\caption{Nonlocal plasmonic eigenvalues for a spherical particle compared with the local quasi-static spectrum. The horizontal axis shows the mode index \(j\), while the vertical axis denotes the corresponding plasmonic eigenvalue \(\varepsilon_j\). In the local theory the eigenvalues accumulate at $-1$. In contrast, since $\sigma_{\text{nonlocal}}$ accumulates at \(1\), only a finite number of negative nonlocal plasmonic eigenvalues exist for any fixed \(h\).}
		\label{fig: eigenvalues}
	\end{figure}
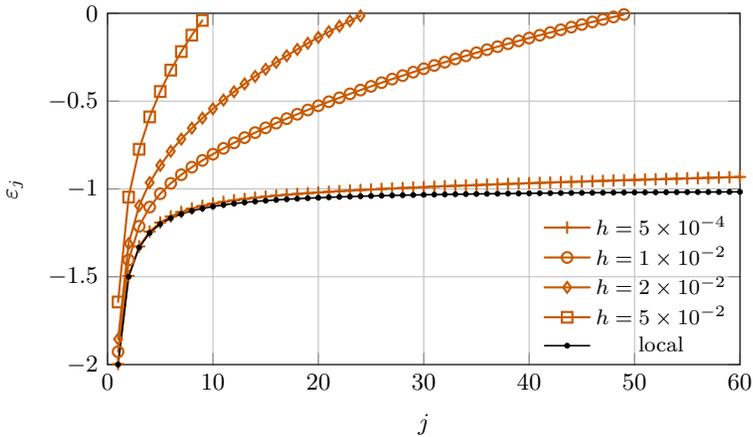

	\subsection{Consistency with known asymptotic results} \label{sec: spectral asymptotics}
	In typical plasmonic regimes one has $h\ll1$, and it is therefore natural to investigate whether, and in what sense, nonlocal plasmonic eigenvalues converge to their local counterparts as $h\to0$. This behaviour is illustrated in Figure~\ref{fig: eigenvalues}, where we observe that the number of negative nonlocal plasmonic eigenvalues increases as $h$ decreases and that these eigenvalues converge non-uniformly to the corresponding local plasmonic eigenvalues. 
	\begin{remark}
		The non-uniformity of this convergence reflects the singularly perturbed nature of the nonlocal model. Indeed, the interior potential $u$ satisfies the fourth-order equation
		\begin{equation*}
			\frac{h^2}{z^2}\Delta^2 u+\Delta u=0
			\qquad\text{in }D,
		\end{equation*}
		which may be viewed as a high-order perturbation of the local second-order problem.
	\end{remark}
	
	First-order asymptotic corrections to the nonlocal plasmonic eigenvalues were derived in \cite{Schnitzer:16b} using matched asymptotic techniques. In this section, we show that our formalism naturally recovers these corrections. As our focus here is on providing a consistency check, we work at a formal level and assume that the corresponding eigenvalues converge as $h\to0$. A rigorous proof of this spectral convergence would likely require tools beyond standard Rouché-type arguments \cite{Ammari_book:09}, and pursuing such a justification would take us beyond the scope of the present paper.
	
	\medskip
	
	Let $\partial D$ be smooth, $z \in \mathbb{C}$ with $\Im z > 0$, $h > 0$, and $\psi \in C^{1}(\partial D)$. Standard asymptotic evaluation of the layer potential integrals shows (see Lemma~\ref{lem: asymptotics S_D^k and K_D^k}) that as \(h\to0\), and for $\psi \in C^{1}(\partial D)$,
	\[
	\SS^{\frac{z}{h}}[\psi](x)
	=
	-\frac{i h}{2z}\,\psi(x)+O(h^2),
	\qquad x\in\partial D,
	\]
	and
	\[
	\KK^{*,\frac{z}{h}}[\psi](x)=O(h),
	\qquad x\in\partial D.
	\]
	This implies that for \(z\) such that \(\Im z>0\) and for any regular density \(\psi\),
	\begin{equation}\label{eq: singular perturbed expansion}
		-2\Lambda_h(z)[\psi]
		=
		\left(z^2I+ \frac{1}{2}I + \mathcal{K}^*\right)[\psi]
		+\frac{i h}{z}\left(-\frac14 I+(\mathcal{K}^*)^2\right)
		(\SS)^{-1}[\psi]
		+O(h^2).
	\end{equation}
	
	\begin{remark}
		The asymptotic expansion \eqref{eq: singular perturbed expansion} shows that the operator \(2\Lambda_h(z)\) converges strongly, as \(h \to 0\), to
		$
		z^2I+ \frac{1}{2}I + \mathcal{K}^*.
		$
		This convergence cannot, however, hold in operator norm. Indeed, norm convergence would imply that \(\SS_1^{\frac{z}{h}} \to \SS\) in norm, which is impossible, since \(\SS\) is invertible whereas \(\SS_1^{\frac{z}{h}}\) is compact for every \(h\). This failure of norm convergence reflects the singularly perturbed nature of the nonlocal model.
	\end{remark}
	
	Up to errors of order \(O(h^2)\), the surface plasmon resonance problem may therefore be reduced to finding non-trivial \(\psi\) and \(z\) such that
	\[
	\Lambda(z)[\psi]=0,
	\]
	where
	\[
	\Lambda(z)
	=
	z^2I+\tfrac12 I+\mathcal{K}^*
	+\frac{ih}{z}\left(-\frac14 I+(\mathcal{K}^*)^2\right)
	(\SS)^{-1}.
	\]
	
	Let \(\psi^{\mathrm{loc}}\) be a local plasmonic eigenmode associated with the local plasmonic eigenvalue \(\varepsilon^{\mathrm{loc}}\); both \(\psi^{\mathrm{loc}}\) and \(\varepsilon^{\mathrm{loc}}\) are real, and \(\varepsilon^{\mathrm{loc}}<0\). Define
	\[
	z^{\mathrm{loc}}
	=
	\sqrt{\frac{\varepsilon^{\mathrm{loc}}}{1-\varepsilon^{\mathrm{loc}}}}.
	\]
	Then
	\[
	\left((z^{\mathrm{loc}})^2I + \frac12 I+\KK^*\right)[\psi^{\mathrm{loc}}]
	= 0.
	\]
	
	Assuming that \((z^{\mathrm{loc}})^2\) is a simple eigenvalue, and assuming the existence of \(z\) near \(z^{\mathrm{loc}}\) and \(\psi\) near \(\psi^{\mathrm{loc}}\) such that \(\Lambda(z)[\psi]=0\), we seek a expansion of the form
	\[
	\psi=\psi^{\mathrm{loc}}+h\psi_1+O(h^2),
	\qquad
	z=z^{\mathrm{loc}}+ah+O(h^2).
	\]
	
	Collecting terms of order \(h\) in the equation \(\Lambda(z)[\psi]=0\) yields
	\begin{equation}\label{eq_int_1}
		\left((z^{\mathrm{loc}})^2I
		+\frac12 I+\mathcal{K}^*\right)[\psi_1]
		+2az^{\mathrm{loc}}\psi^{\mathrm{loc}}
		+i(z^{\mathrm{loc}})^{-1}\left(-\frac14 I+(\mathcal{K}^*)^2\right)
		(\SS)^{-1}[\psi^{\mathrm{loc}}]
		=0.
	\end{equation}
	
	Multiplying \eqref{eq_int_1} by \(\SS[\psi^{\mathrm{loc}}]\), integrating over \(\partial D\), and using the symmetrisation identity \eqref{eq: symmetrisation identity}, we obtain
	\[
	2a z^{\mathrm{loc}}
	\left\langle\SS[\psi^{\mathrm{loc}}]\,,\,\psi^{\mathrm{loc}}\right\rangle_{L^2(\partial D)}
	+
	iz^{\mathrm{loc}}((z^{\mathrm{loc}})^2+1)
	\|\psi^{\mathrm{loc}}\|^2_{L^2(\partial D)}
	=0,
	\]
	which gives
	\begin{equation}\label{eq: a}
		a
		=
		-i\frac{(z^{\mathrm{loc}})^2+1}{2}
		\frac{\|\psi^{\mathrm{loc}}\|^2_{L^2(\partial D)}}	{\left\langle\SS[\psi^{\mathrm{loc}}]\,,\,\psi^{\mathrm{loc}}\right\rangle_{L^2(\partial D)}}.
	\end{equation}
	
	Using the relation
	\(
	\varepsilon=1-\frac{1}{z^2+1},
	\)
	we recover the first-order asymptotic correction to the plasmonic eigenvalue \cite{Schnitzer:16b}:
	\begin{align*}
		\varepsilon
		&=
		1-\frac{1}{(z^{\mathrm{loc}})^2+1+2haz^{\mathrm{loc}}+O(h^2)}\\
		&=
		\varepsilon^{\mathrm{loc}}
		-
		h\sqrt{\varepsilon^{\mathrm{loc}}(\varepsilon^{\mathrm{loc}}-1)}
		\frac{\|\psi^{\mathrm{loc}}\|^2_{L^2(\partial D)}}	{\left\langle\SS[\psi^{\mathrm{loc}}]\,,\,\psi^{\mathrm{loc}}\right\rangle_{L^2(\partial D)}}
		+O(h^2).
	\end{align*}
	
	\section{Resonance excitation}\label{sec: modal expansion}
	\subsection{Modal expansion}
	We now return to the operator $\tilde{\Lambda}_h(z)$, defined in \eqref{eq: tilde{Lambda}_h(z)}. From Theorem~\ref{thm: main result} we know that $\tilde{\Lambda}_h(z)$ is an analytic family of Fredholm operators of index zero that is invertible for at least one $z\in\mathbb{C}$. Hence, by the analytic Fredholm theorem, the map \[ z \mapsto \tilde{\Lambda}^{-1}_h(z) \] is a meromorphic operator-valued function on $\mathbb{C}$ with a discrete set of isolated and finite order poles that we denote by
	$
	\Omega\subset\mathbb{C}.
	$
	Note that $\Omega$ contains poles associated with both nonlocal plasmonic resonances and scattering resonances associated with the single layer potential. 
	
		Assuming that the poles are simple, the Laurent expansion of $\tilde{\Lambda}^{-1}_h(z)$ in \(z\in B_j\setminus\{z_j\}\) takes the form \[ 
	\tilde{\Lambda}^{-1}_h(z) = \frac{\mathcal{P}_j}{z - z_j} + \mathcal{H}_j(z), 
	\] where $\mathcal{H}_j(z)$  is a \(z_j\)-dependent holomorphic operator-valued function in an open neighbourhood \(B_j\) of \(z_j\) such that $B_j\cap\Omega=\{z_j\}$. $\mathcal{P}_j$ is a finite-rank operator (the residue at $z_j$) that defines the resonant (modal) subspace associated with the resonance $z_j$. In particular, for a given $\varphi\in H^{-\frac12}(\partial D)$ the residue projector takes the form
	\[ 
	\mathcal{P}_{j} \varphi = \sum_{\psi_j\in \ker \tilde{\Lambda}_h(z_j)} \frac{ \langle \psi_j^*, \varphi\rangle_{L^2(\partial D)} }{ \left\langle \psi_j^*, \frac{d\tilde{\Lambda}_h}{dz}(z_j)\psi_j\right\rangle_{L^2(\partial D)} } \,\psi_j, 
	\]
	where the function $\psi_j^*\in \ker \tilde{\Lambda}_h^*(z_j),$ can be picked uniquely by imposing usual bi-orthogonality relations with the functions in $\ker \tilde{\Lambda}_h(z_j)$.
	\begin{remark}
	For a simple pole we have that for each 
		$ 
		\psi_j\in \ker \tilde{\Lambda}_h(z_j) 
		$ 
		it holds 
		$ 
		\left\langle \psi_j^*\,,\, \frac{d\tilde{\Lambda}_h}{dz}(z_j)\psi_j\right\rangle_{L^2(\partial D)} \neq 0.
		$ 
	\end{remark}
	\begin{remark} 
		In the case of a sphere, the operator $\tilde{\Lambda}_h(z)$ (and $\Lambda_h(z)$) can be diagonalised in the basis of spherical harmonic functions. One can thereby show that all poles in the case of a sphere are simple. We show these calculations in Appendix~\ref{sec: Appendix sphere}.
	\end{remark} 

	Consequently, the solution 
	\[ 
	\psi(z) = \tilde{\Lambda}^{-1}_h(z)\left[\frac{\partial v^{\text{ext}}}{\partial \nu}\right] 
	\]
	admits the following local meromorphic expansion for \(z\in B_j\setminus\{z_j\}\):
	\begin{equation} \label{eq: mod expansion psi}
		\psi(z) = \frac{1}{z - z_j}\sum_{\psi_j\in \ker \tilde{\Lambda}_h(z_j)} \frac{ \langle \psi_j^*, \frac{\partial v^{\text{ext}}}{\partial \nu}\rangle_{L^2(\partial D)} }{ \left\langle \psi_j^*, \frac{d\tilde{\Lambda}_h}{dz}(z_j)\psi_j\right\rangle_{L^2(\partial D)} } \,\psi_j + \mathcal{H}_j(z)\left[\frac{\partial v^{\text{ext}}}{\partial \nu}\right]. 
	\end{equation} 
	This leads to the following modal expansion result for the exterior potential $v$. (similar meromorphic expansions hold for the bulk fields $\rho$ and $u$).
	\begin{theorem}\label{thm: modal expansion} Let \(\Sigma \subset \Omega\) be a finite set of poles of $\tilde{\Lambda}^{-1}_h(z)$ that are assumed to be simple. For $z$ in a neighbourhood of $\Sigma$, the exterior potential $v$, solution to the nonlocal system \eqref{eq:nonlocal scattering problem}, admits the following modal expansion form,
		\begin{align*}
			v(z) = \sum_{z_j\in\Sigma} \sum_{\psi_j\in \ker \tilde{\Lambda}_h(z_j)}  \frac{c_j}{z-z_j}v_j + \mathcal{V}_{\Sigma}(z)\left[\frac{\partial v^{\text{ext}}}{\partial \nu}\right] + v^{\text{ext}} ,
		\end{align*} 
		where $\mathcal{V}_{\Sigma}(z)$ is a holomorphic operator-valued map, $v_j$ is a nonlocal plasmonic mode exterior potential (satisfying the nonlocal plasmonic eigenvalue problem \eqref{eq:nonlocal plasmonic eigenvalue problem}), and 
		\[ 
		c_j = \frac{ \langle \psi_j^*, \frac{\partial v^{\text{ext}}}{\partial \nu}\rangle_{L^2(\partial D)} }{ \left\langle \psi_j^*, \frac{d\tilde{\Lambda}_h}{dz}(z_j)\psi_j\right\rangle_{L^2(\partial D)}},
		\]
		with $\psi_j^*\in  \ker \tilde{\Lambda}^*_h(z_j)$ bi-orthogonal to $\psi_j$.
		
	\end{theorem} 
	\begin{proof} From the layer potential representation established in Section~\ref{sec: layer potential formulation}, we have
		\[
		v(z)=\SS[q(z)]+v^{\mathrm{ext}},
		\]
		where
		\[
		q(z)
		=
		-\left(\frac12 I+\KK^*\right)^{-1}
		\left(-\frac12 I+\KK^{*,\frac zh}\right)[\psi(z)]
		-
		\left(\frac12 I+\KK^*\right)^{-1}\left[\frac{\partial v^{\mathrm{ext}}}{\partial \nu}\right].
		\]
		Choose an open neighbourhood \(B_j\) of \(z_j\) such that
		$
			B_j\cap \Omega=\{z_j\}.
		$
		Since \(z\mapsto \KK^{*,z/h}\) is analytic, it follows that
		\[
		z\mapsto \left(-\frac12 I+\KK^{*,\frac zh}\right)
		\]
		is analytic on \(B_j\). Therefore, using \eqref{eq: mod expansion psi} to expand \(\psi(z)\), we obtain, for \(z\in B_j\setminus\{z_j\}\),
		\begin{align*}
			\left(-\frac12 I+\KK^{*,\frac zh}\right)[\psi(z)]
			&=
			\frac{1}{z-z_j}
			\sum_{\psi_j\in \ker \tilde{\Lambda}_h(z_j)}
			\frac{\left\langle \psi_j^*, \frac{\partial v^{\mathrm{ext}}}{\partial \nu}\right\rangle_{L^2(\partial D)}}
			{\left\langle \psi_j^*, \frac{d\tilde{\Lambda}_h}{dz}(z_j)\psi_j\right\rangle_{L^2(\partial D)}}
			\left(-\frac12 I+\KK^{*,\frac{z_j}{h}}\right)[\psi_j] \\
			&\qquad
			+\mathcal{W}_j(z)\left[\frac{\partial v^{\mathrm{ext}}}{\partial \nu}\right],
		\end{align*}
		where \(\mathcal{W}_j(z)\) is holomorphic on \(B_j\). Substituting this into the expression for \(q(z)\), we obtain
		\begin{align*}
			q(z)
			&=
			-\frac{1}{z-z_j}
			\sum_{\psi_j\in \ker \tilde{\Lambda}_h(z_j)}
			\frac{\left\langle \psi_j^*, \frac{\partial v^{\mathrm{ext}}}{\partial \nu}\right\rangle_{L^2(\partial D)}}
			{\left\langle \psi_j^*, \frac{d\tilde{\Lambda}_h}{dz}(z_j)\psi_j\right\rangle_{L^2(\partial D)}}
			\left(\frac12 I+\KK^*\right)^{-1}
			\left(-\frac12 I+\KK^{*,\frac{z_j}{h}}\right)[\psi_j] \\
			&\qquad
			+\mathcal{Q}_j(z)\left[\frac{\partial v^{\mathrm{ext}}}{\partial \nu}\right],
		\end{align*}
		for some holomorphic operator-valued function \(\mathcal{Q}_j(z)\) on \(B_j\). Applying \(\SS\) and recalling the definition of the nonlocal plasmonic mode \(v_j\),
		\[
		v_j
		=
		-\SS\left(\frac12 I+\KK^*\right)^{-1}
		\left(-\frac12 I+\KK^{*,\frac{z_j}{h}}\right)[\psi_j],
		\]
		we arrive at the local representation
		\[
		v(z)
		=
		\frac{1}{z-z_j}
		\sum_{\psi_j\in \ker \tilde{\Lambda}_h(z_j)}
		\frac{\left\langle \psi_j^*, \frac{\partial v^{\mathrm{ext}}}{\partial \nu}\right\rangle_{L^2(\partial D)}}
		{\left\langle \psi_j^*, \frac{d\tilde{\Lambda}_h}{dz}(z_j)\psi_j\right\rangle_{L^2(\partial D)}}
		\,v_j
		+
		\mathcal{V}_j(z)\left[\frac{\partial v^{\mathrm{ext}}}{\partial \nu}\right]
		+
		v^{\mathrm{ext}},
		\]
		for $z\in B_j\setminus\{z_j\}$, where \(\mathcal{V}_j(z)\) is holomorphic on \(B_j\).
		
 By the local representation above, \(v(z)\) admits, in each \(B_j\), a decomposition into a principal part at \(z_j\) and a holomorphic remainder. Subtracting the sum of all principal parts over \(\Sigma\) therefore removes all singularities of \(v(z)\) in \(U := \bigcup_{z_j \in \Sigma} B_j\). It follows that the remainder extends to a holomorphic operator-valued function on \(U\), which we denote by \(\mathcal{V}_\Sigma(z)\). This yields the claimed expansion.
	\end{proof}
	
	\medskip
	
	When applying this theorem we must take into account that 
	$$
	z(\omega)=\sqrt{\frac{\varepsilon(\omega)}{1-\varepsilon(\omega)}},
	$$ 
	where $\varepsilon(\omega)$ is determined by the Drude model \eqref{eq:Drude}. This implies that $\Im(z^{2}(\omega))>0$ and, therefore, $z(\omega)$ must lie in the first quadrant. On the other hand, poles arising from $\sigma_{\text{nonlocal}}$ lie either on the positive imaginary axis (corresponding to surface modes with $\varepsilon<0$) or on the positive real axis (corresponding to bulk modes with $\varepsilon>0$), whereas scattering resonances lie in the lower half-plane. Consequently, in physically relevant regimes, only the nonlocal poles can resonate within this modal representation. Moreover, even if $z$ were to approach a scattering resonance, such poles correspond to fields that vanish in the interior domain and therefore do not contribute to the modal representation of the bulk quantities $\rho$, $u$, and $v$.
	
	We emphasise that our derivation relies only on the Fredholm analyticity of $\tilde{\Lambda}_h(z)$. 
	Therefore, contrary to the local case, the modal expansion remains valid for any Lipschitz domain, thus yielding a meaningful resonant modal expansion in regimes where the local quasi-static theory breaks down. In particular, even in the presence of corners or conical points, one obtains well-defined nonlocal plasmonic modes, which is consistent with previous numerical investigations \cite{wiener:2012}. 
	
	\subsection{Far field excitation}
	We now consider the particular case where the nanoparticle is excited by a far-field source, so that $f = 0$ in \eqref{eq:nonlocal scattering problem} and $g$ is such that 
	$
	v^{\mathrm{ext}}(x) = \hat{\imath}\cdot x.
	$
	This setting corresponds to an incoming plane wave in the quasi-static approximation. Such excitation typically couples to a finite (small) number of surface modes. 
	In this context, combining the modal expansion of Theorem~\ref{thm: modal expansion} with the perturbation formula of nonlocal eigenvalues discussed in Section~\ref{sec: spectral asymptotics}, yields the following result, pertinent to far-field excitations.
	\begin{proposition}\label{prop: resonance shift}
		Assume that only finitely many modes in a set $\Sigma$ are excited, that these correspond to simple one-dimensional poles, and that the associated nonlocal resonances converge to the corresponding local resonances as $h\to 0$. Then, as $h\to 0$,
		\[
		v(z) =
		\sum_{z_j\in\Sigma}
		\frac{c_j^{\text{loc}}}{z - (z_j^{\text{loc}} + a_j h) + O(h^2)}\, v_j^{\text{loc}}
		+
		\mathcal{V}_{\Sigma}(z)
		\!\left[
		\frac{\partial v^{\mathrm{ext}}}{\partial \nu}
		\right]
		+
		v^{\mathrm{ext}}
		+
		O(h),
		\]
		where
		\[
		c_j^{\text{loc}}
		=
		\frac{
			\left\langle (\psi_j^{\text{loc}})^*,
			\frac{\partial v^{\mathrm{ext}}}{\partial \nu}
			\right\rangle_{L^2(\partial D)}
		}{
			2\left\langle (\psi_j^{\text{loc}})^*, \psi_j^{\text{loc}} \right\rangle_{L^2(\partial D)}
		},
		\]
		and
		\[
		a_j
		=
		-i\frac{(z_j^{\text{loc}})^2+1}{2}
		\frac{\|\psi^{\mathrm{loc}}\|^2_{L^2(\partial D)}}	{\left\langle\SS[\psi^{\mathrm{loc}}]\,,\,\psi^{\mathrm{loc}}\right\rangle_{L^2(\partial D)}}.
		\]
	\end{proposition}
	Proposition~\ref{prop: resonance shift} gives a modal expansion for the scattered field that connects the well-known blue shift for nonlocal plasmonic resonances with spectral shifts that can be estimated, in general geometries, using perturbation formulas first derived in \cite{Schnitzer:16b} and recovered in Section~\ref{sec: spectral asymptotics} via the layer potential framework. The resonance shift can be appreciated in standard optical observables such as the absorption cross-section \cite{Bohren:Book}, which is extracted from the dipolar far-field asymptotic of the scattered field:
	\[
	\mathrm{Absorption}
	:=
	\frac{\omega}{\omega_p}\,\Im \mu,
	\]
	where the dipole moment coefficient $\mu$ is determined from the expansion
	\[
	\mu \frac{\hat{\imath}\cdot x}{|x|^3}
	\sim
	v(x) - v^{\mathrm{ext}}(x),
	\qquad |x|\to\infty.
	\]
	
	In the case of a sphere, $\mu$ can be computed explicitly \cite{Schnitzer:16b} and may also be recovered directly from the calculations in Appendix~\ref{sec: Appendix sphere}. 
	The resulting absorption spectrum for a sphere is shown in Figure~\ref{fig: far field sphere} where one can observe the  blue shift of the dipolar resonance relative to the local quasi-static model.
	\begin{figure}[h]
		\centering
		\begin{tikzpicture}	
			\begin{semilogyaxis}[
				width=10cm,
				height=6cm,
				xmin=0.3, xmax=1.5,
				ymin=0, ymax=100,
				xlabel={$\omega/\omega_p$},
				ylabel={Absorption},
				grid=both,
				ticklabel style={font=\small},
				label style={font=\small},
				legend style={font=\small, at={(1,0.98)}, anchor=north east, draw=none, fill=none},
				line cap=round,
				]
				
				\addplot[color={rgb,255:red,200; green,90; blue,0}, very thick] table[col sep=tab] {ffext_nonlocal.dat};
				\addlegendentry{nonlocal}
				
				\addplot[semithick] table[col sep=tab] {ffext_local.dat};
				\addlegendentry{local}
				
			\end{semilogyaxis}
			
		\end{tikzpicture}
		\caption{Absorption spectrum for a spherical nanoparticle under far-field excitation. We use $h =  2\times10^{-2}$ for the nonlocal parameter and $\gamma/\omega_p = 1\times 10^{-1}$ for the Drude model. The nonlocal model exhibits the characteristic blue shift of the dipolar plasmon resonance relative to the local quasi-static model (black curve). Additional weaker peaks above the plasma frequency correspond to longitudinal bulk modes supported by the nonlocal model.}
		\label{fig: far field sphere}
	\end{figure}
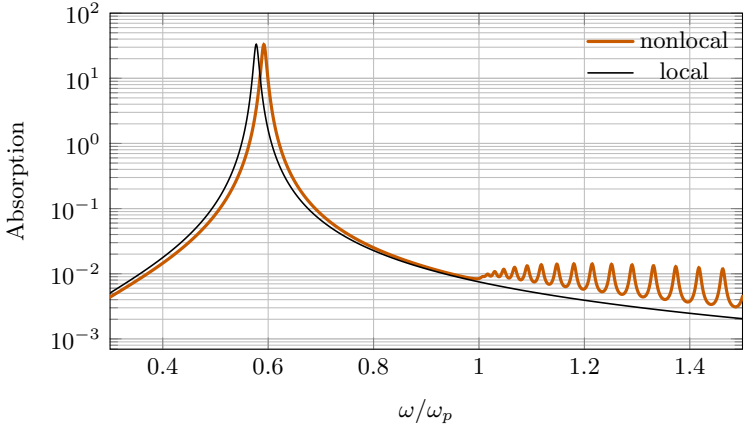
	
	One can also observe several peaks above the plasma frequency $\omega_p$, which are orders of magnitude weaker than the dominant dipolar surface plasmon. These arise from the coupling of far-field excitation to longitudinal bulk modes via their non-trivial boundary trace interacting with $\frac{\partial v^{\mathrm{ext}}}{\partial \nu}$. This behaviour is heuristically explained by Theorem~\ref{thm: modal expansion}, since the coupling coefficients $c_j$ are surface integrals, whereas bulk modes are interior-localised and therefore weakly coupled to the boundary excitation.
	
	\subsection{Near field excitation}
	\begin{figure}[h]
	\centering
	\begin{tikzpicture}[scale=0.7, baseline=(current bounding box.center)]		
		\shade[ball color = gray!20, opacity = 0.4] (0,0) circle (2cm);
		\draw (0,0) circle (2cm);
		\draw (-2,0) arc (180:360:2 and 0.6);
		\draw[dashed] (2,0) arc (0:180:2 and 0.6);
		
		\node at (4,-0.08) {\textcolor{red}{\Large{$\leftrightarrow$}}};
		
		\draw[|-|, dashed] (2.05,0) -- node[midway, above] {$d$} (4,0);
	\end{tikzpicture}
	\hspace{0.5cm}	
	\begin{tikzpicture}[baseline=(current bounding box.center)]
		\begin{groupplot}[
			group style={
				group size=1 by 3,   
				vertical sep=0cm, 
			},
			width=7.5cm,
			height=4.8cm,
			ymin=0, 
			xmin=0.4, xmax=1,
			xtick align=inside,
			ymode=log,
			grid=both,
			ticklabel style={font=\small},
			label style={font=\small},
			legend style={font=\small, at={(1,0.98)}, anchor=north east, draw=none, fill=none},
			]
			
			\nextgroupplot[
			title = {\footnotesize{$\left|\dfrac{\partial (v -  v^{\text{ext}})}{\partial r}(x_d)\right|$}},
			ymax=1000,
			ylabel={$d/R = 0.1$},
			ytick={1e2, 1e3},
			xticklabels=\empty, 
			]
			\addplot[color={rgb,255:red,200; green,90; blue,0}, very thick] table[col sep=tab] {nfext_nonlocal_1dot1.dat};
			\addlegendentry{nonlocal}
			
			\addplot[semithick] table[col sep=tab] {nfext_local_1dot1.dat};
			\addlegendentry{local}
			
			\nextgroupplot[
			ymax=100,
			ylabel={$d/R = 0.3$},
			ytick={1e0,1e1},
			xticklabels=\empty, 
			]
			\addplot[color={rgb,255:red,200; green,90; blue,0}, very thick] table[col sep=tab] {nfext_nonlocal_1dot3.dat};
			\addlegendentry{nonlocal}
			
			\addplot[semithick] table[col sep=tab] {nfext_local_1dot3.dat};
			\addlegendentry{local}
			
			\nextgroupplot[
			ymax=10,
			xlabel={$\omega/\omega_p$},
			ylabel={$d/R = 0.5$},
			ytick={1e-1,1e0},
			]
			\addplot[color={rgb,255:red,200; green,90; blue,0}, very thick] table[col sep=tab] {nfext_nonlocal_1dot5.dat};
			\addlegendentry{nonlocal}
			
			\addplot[semithick] table[col sep=tab] {nfext_local_1dot5.dat};
			\addlegendentry{local}
			
		\end{groupplot}
	\end{tikzpicture}
	\caption{Near–field excitation of a spherical nanoparticle by a point dipole located at distance $d$ from the surface; we use $h =  2\times10^{-2}$ for the nonlocal parameter and $\gamma/\omega_p = 1\times 10^{-1}$ for the Drude model. The panels show the magnitude of the radial component of the scattered electric field at the dipole position for different relative distances $d/R$, where $R$ is the radius of the sphere. In the local model the response develops a broad collective resonance as $d$ decreases due to the accumulation of surface plasmon modes. In contrast, the nonlocal model exhibits well-separated resonances, reflecting the finite number of surface modes and the absence of spectral clustering.}
	\label{fig: near field sphere}
\end{figure}
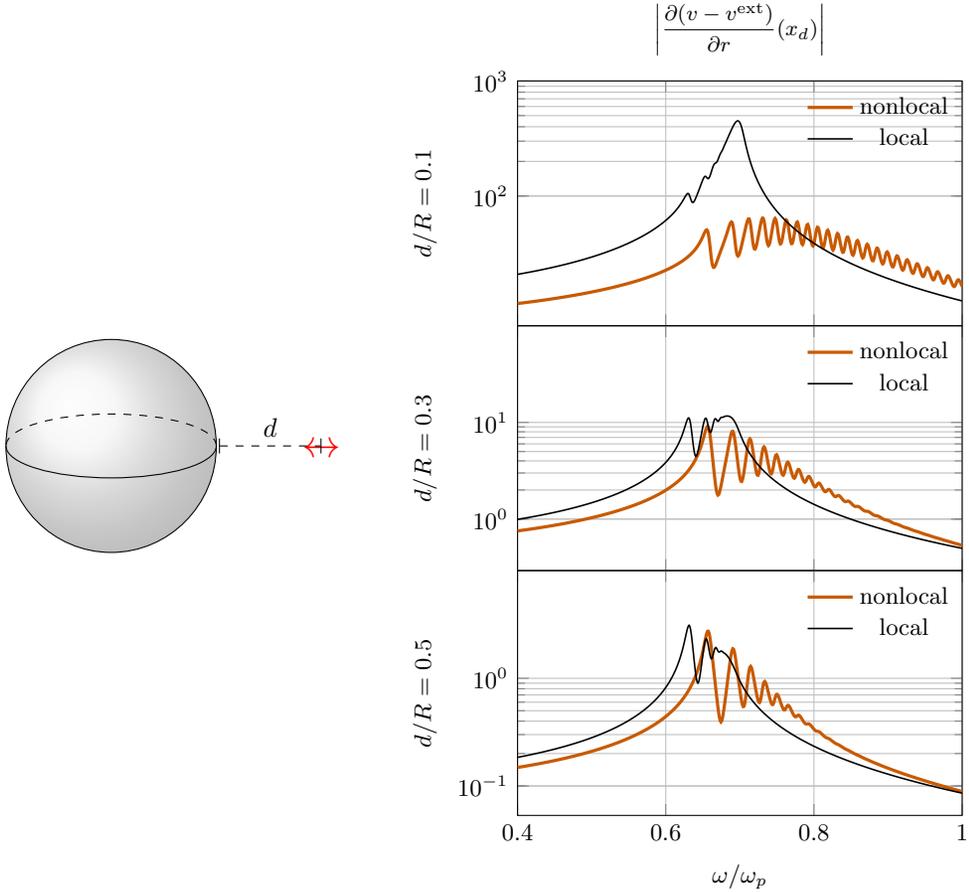
	Finally, we consider the case where the nanoparticle is excited by a near–field source, so that $g = 0$ in \eqref{eq:nonlocal scattering problem}. In particular, we consider the excitation generated by a point dipole located at
	$
	x_d\in \mathbb R^3\setminus \overline D,
	$
	with dipole moment $p$ oriented along the outward normal direction and at a distance $
	d=\mathrm{dist}(x_d,\partial D)
	$
	from the boundary. In the quasi--static regime, this excitation is represented by the dipole-type distribution
	$
	f(x)=p\cdot \nabla_x \delta(x-x_d),
	$
	where \(\delta\) is the Dirac distribution. The resulting external potential is then given by
	\[
	v^{\mathrm{ext}}(x)=p\cdot \nabla_x \Gamma^0(x-x_d).
	\]
	This type of excitation produces boundary data
	$
	\frac{\partial v^{\rm ext}}{\partial \nu}
	$
	that is strongly localised near the closest point to $\partial D$. As $d\to 0$, this forcing becomes increasingly singular and therefore contains progressively higher spatial oscillations along the surface.
	
	A typical quantity of interest for a dipole excitation is the radial component of the scattered electric field at the dipole position. In the local quasi–static case one has the modal representation
	\[
	\dfrac{\partial (v -  v^{\text{ext}})}{\partial r}(x_d)
	=
	(\varepsilon(\omega)-1)\sum_{j=1}^{\infty} 
	\frac{(\varepsilon^{\text{loc}}_j-1)2c_j^{\text{loc}}(d)}{\varepsilon(\omega)-\varepsilon^{\text{loc}}_j}
	\frac{\partial v_j^{\text{loc}}}{\partial r}(x_d),
	\]
	where $\varepsilon_j^{\text{loc}}$ denote the local plasmonic eigenvalues and $v_j^{\text{loc}}$ the corresponding local surface modes. For smooth geometries the sequence $\varepsilon_j^{\text{loc}}$ accumulates at $-1$. As $d\to0$, the coefficients $c_j^{\text{loc}}(d)$ no longer decay rapidly with $j$, and the observable therefore samples a dense set of resonances frequencies for which $\varepsilon(\omega)\approx -1$, producing a pronounced collective resonance effect driven by a high density of states. Instead of isolated peaks, the response develops a broadened and strongly enhanced structure resulting from the coherent contribution of many high-index surface modes. This mechanism, also observed in \cite{Christensen:2014}, is illustrated in the local curves in Figure~\ref{fig: near field sphere}, which serve as a canonical example of this phenomenon, that is general to any smooth geometry.

	In the nonlocal model, the situation is fundamentally different. As first observed numerically in \cite{Christensen:2014}, and as shown in Figure~\ref{fig: near field sphere} for the sphere, multipolar surface resonances remain sharply distinguishable and no clustering occurs, even as $d$ decreases. Our results show that this behaviour is not specific to the sphere but is a general consequence of Theorem~\ref{thm: main result} together with Theorem~\ref{thm: modal expansion}. Indeed, in this case, we may write
	\[
	\dfrac{\partial (v -  v^{\text{ext}})}{\partial r}(x_d)
	=
	\sum_{z_j\in\Sigma}
	\frac{c_j(d)}{z(\omega)-z_j}
	\frac{\partial v_j}{\partial r}(x_d)
	+
	\frac{\partial}{\partial r}
	\mathcal{V}_{\Sigma}(z)
	\Big[
	\frac{\partial v^{\text{ext}}}{\partial \nu}
	\Big](x_d),
	\]
	where $\Sigma$ is any finite set of surface–plasmon characteristic values. By Theorem~\ref{thm: main result}, the number of surface modes is finite in the relevant spectral region, so $\Sigma$ in Theorem~\ref{thm: modal expansion} may be chosen to contain all of them. Hence, even as $x_d\to\partial D$, the response is governed by a finite number of well–defined poles which, in frequency, are distributed throughout the interval below the plasma frequency, with no accumulation phenomenon. 
	
	\section{Concluding remarks} \label{sec: conclusion}
	In this work, we have placed the modal analysis of nonlocal plasmonics on a rigorous mathematical foundation. By recasting the quasi-static hydrodynamic Drude model as a boundary integral operator pencil and invoking analytic Fredholm theory, we obtained a convenient spectral framework for the definition and analysis of nonlocal plasmonic eigenvalues and modes. To the best of our knowledge, this is the first fully rigorous spectral characterisation of plasmonic resonances in the presence of spatial dispersion.
	
	A central outcome of our analysis is that hydrodynamic nonlocality acts as a spectral regularisation mechanism, restoring a well-defined modal structure even in non-smooth geometries. In contrast to the local problem, where the Neumann–Poincaré operator loses compactness on Lipschitz domains, the nonlocal formulation always yields an analytic Fredholm operator pencil with a discrete spectrum. As a result, only finitely many surface modes persist, and the energy-divergent behaviour associated with singular geometries in the quasi-static limit is eliminated.
	
	From an operator-theoretic perspective, this regularisation stems from the higher-order nature of the model, which introduces additional boundary conditions and acts as a singular perturbation of the underlying second-order problem. This singularly perturbed structure is directly encoded in the boundary integral formulation, which yields a family of operators depending on the nonlocal parameter \(h\) and which, in smooth domains, converge strongly---but not in norm---to their local counterparts as \(h \to 0\). 
	
	A related technical point concerns the choice of the spectral parameter. In the present formulation, the analytic parameter is $z$, rather than the more natural $z^2$; this choice reflects the use of the outgoing Green function for the longitudinal fields. While a formulation purely in terms of $z^2$ can, in principle, be obtained via symmetrisation of the Green function, such an approach introduces an additional family of artificial resonances collapsing onto the real axis, which are difficult to distinguish from genuine nonlocal eigenvalues. By contrast, the outgoing formulation effectively shifts these non-physical resonances into the complex plane, yielding a clearer spectral separation and ensuring that the physical poles remain distinct and identifiable.
	
	One might alternatively pursue a variational formulation of the nonlocal system \eqref{eq:nonlocal scattering problem} in terms of bulk quantities, which could potentially yield a simpler linear pencil in the variable \(z^2\). Such a formulation may provide complementary insights into the spectral structure, for instance regarding the completeness of the nonlocal modes and the infinitude of the associated eigenvalues. However, the boundary integral formulation adopted here offers several decisive advantages: it reduces the problem to a single operator equation on the interface, which is naturally suited for efficient numerical discretisation. Moreover, the associated operators admit explicit symbolic representations, providing a concrete analytical framework for further rigorous analysis, including spectral convergence as \(h \to 0\) and the analysis of field saturation in singular geometries.
	
	The presence of spectral pathologies in local models has long been associated with singular field behaviour of direct relevance to applications. This connection has motivated substantial mathematical effort devoted to understanding the emergence of continuous spectrum and singular phenomena in second-order elliptic problems on non-smooth domains, including the appearance of essential spectrum and corner-induced singularities \cite{BonnetBenDhia:2013, Perfekt:14, Bonnetier:17, perfekt:2019, Li:2019, perfekt:2021, bonnetbendhia:2021, FARIA:2025, BonnetBenDhia:2025}. The present work highlights a complementary and, we believe, equally compelling direction: the study of higher-order or semiclassical perturbations, which can fundamentally reorganise the spectral structure and suppress these singular features.
	
	
	\bibliographystyle{siamplain}
	\bibliography{references}
	
	\appendix
	
		\section{Solution in a corner}\label{sec: Appendix corner}
	Let \(D\subset\mathbb{R}^2\) be a domain with a corner of opening angle
	\(\theta_0\in(0,2\pi)\), so that for some \(r_0>0\),
	\[
	W:=D\cap B_{r_0}(0)
	=
	\{(r,\theta):0<r<r_0,\ 0<\theta<\theta_0\}.
	\]
	
	The following result is a consequence of Theorem~\ref{thm: main result}. Although that theorem was proved in $\mathbb{R}^3$, the same arguments—up to minor technical modifications in the layer potential theory—extend to the two-dimensional setting.
	\begin{proposition}\label{prop:corner_constants}
		Let \((\rho,u,v)\) be a nonlocal plasmonic mode. Then there exist constants \(C_\rho,C_u\in\mathbb{C}\) such that
		\[
		\rho(r,\theta)\to C_\rho,\qquad
		u(r,\theta)\to C_u,\qquad
		v(r,\theta)\to C_u,
		\qquad\text{as }r\to0.
		\]
	\end{proposition}
	
	\begin{proof}
		Introduce the logarithmic radial variable
		\[
		t=-\log r,
		\qquad r=e^{-t},
		\]
		so that \(t\to+\infty\) corresponds to approaching the corner tip. In the
		variables \((t,\theta)\), the wedge neighbourhood \(W\) is transformed into the
		infinite strip
		\[
		\{(t,\theta):t>T_0,\ 0<\theta<\theta_0\},
		\qquad T_0:=-\log r_0.
		\]
		
		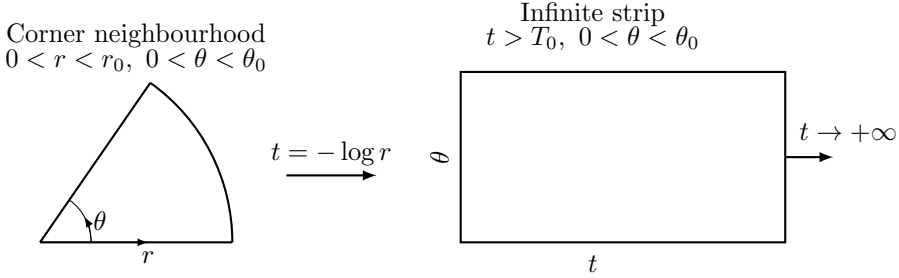
\begin{figure}[h]
			\centering
			\begin{tikzpicture}[scale=0.8,>=latex]
				
				\begin{scope}
					\draw[thick] (0,0) -- (3.2,0);
					\draw[thick] (0,0) -- ({3.2*cos(55)},{3.2*sin(55)});
					\draw[thick] (3.2,0) arc[start angle=0,end angle=55,radius=3.2];
					
					\draw[->] (0,0) -- (1.8,0) node[right,below] {$r$};
					
					\draw (0.85,0) arc[start angle=0,end angle=55,radius=0.85];
					\draw[->] (0.85,0) arc[start angle=0,end angle=30,radius=0.85];
					\node at (1,0.35) {$\theta$};
					
					\node[align=center] at (1.6,3.2)
					{Corner neighbourhood\\[-1mm]$0<r<r_0,\ 0<\theta<\theta_0$};
				\end{scope}
				
				\draw[->,thick] (4.1,1.1) -- (5.6,1.1) node[midway,above] {$t=-\log r$};
				
				\begin{scope}[xshift=7.0cm]
					\draw[thick] (0,0) rectangle (5.4,2.8);
					
					\node at (2.2,-0.35) {$t$};
					\node[rotate=90] at (-0.35,1.4) {$\theta$};
					
					\draw[->,thick] (5.4,1.4) -- (6.2,1.4);
					\node at (6.45,1.8) {$t\to+\infty$};
					
					\node[align=center] at (2.2,3.55)
					{Infinite strip\\[-1mm]$t>T_0,\ 0<\theta<\theta_0$};
				\end{scope}
				
			\end{tikzpicture}
			\caption{The change of variables \(t=-\log r\) maps the wedge neighbourhood \(0<r<r_0,\ 0<\theta<\theta_0\) to the infinite strip \(t>T_0,\ 0<\theta<\theta_0\), where \(T_0=-\log r_0\).}
			\label{fig:wedge_to_strip}
		\end{figure}
		
		For \(w \in\{\rho,u,v\}\), define
		\[
		\hat w(t,\theta):=w(e^{-t},\theta).
		\]
		In polar coordinates \((r,\theta)\), the Laplacian is
		\[
		\Delta=\partial_r^2+\frac1r\partial_r+\frac1{r^2}\partial_\theta^2,
		\]
		and a direct computation gives
		\[
		\Delta w(r,\theta)
		=
		\frac1{r^2}\big(\hat w_{tt}+\hat w_{\theta\theta}\big)
		=
		e^{2t}\big(\partial_t^2+\partial_\theta^2\big)\hat w.
		\]
		
		The interior equations of the nonlocal eigenvalue problem therefore become
		\[
		(\partial_t^2+\partial_\theta^2)\hat\rho
		=
		\lambda e^{-2t}\hat\rho,
		\qquad
		(\partial_t^2+\partial_\theta^2)\hat u
		=
		-\frac{e^{-2t}}{h}\hat\rho,
		\]
		in \(\{(t,\theta):t>T_0,\ 0<\theta<\theta_0\}\), where
		\[
		\lambda=\frac{\varepsilon}{(\varepsilon-1)h^2},
		\]
		while
		\[
		(\partial_t^2+\partial_\theta^2)\hat v=0
		\qquad
		\text{in }\{(t,\theta):t>T_0,\ \theta_0<\theta<2\pi\}.
		\]
		
		The wedge boundary corresponds to the rays \(\theta=0\) and \(\theta=\theta_0\).
		On these boundaries the outward normal is purely angular, so
		\(\partial_\nu=r^{-1}\partial_\theta\), up to orientation. Hence, after
		multiplication by \(r=e^{-t}\), the transmission conditions become
		\(t\)-independent boundary conditions on the strip edges:
		\[
		u=v,
		\qquad
		\partial_\theta u=\partial_\theta v=-h\,\partial_\theta \rho
		\quad\text{on }\theta=0,\theta_0,2\pi,\quad t>T_0.
		\]
		
		Moreover, a direct computation shows that
		\[
		\int_W |\nabla w|^2\,dx
		=
		\int_{T_0}^{\infty}\int_I
		\big(|\partial_t \hat w|^2+|\partial_\theta \hat w|^2\big)\,d\theta\,dt,
		\]
		and
		\[
		\int_W |w|^2\,dx
		=
		\int_{T_0}^{\infty}\int_I e^{-2t}|\hat w|^2\,d\theta\,dt,
		\]
		where \(I=(0,\theta_0)\) for \(\rho,u\), and similarly \(I=(\theta_0,2\pi)\) for
		\(v\). From Theorem~\ref{thm: main result} we have that 	
		\[
		\rho,u\in H^1(W),\qquad v\in H^1(B_{r_0}(0)\setminus \overline W),
		\] 
		which implies
		\[
		\partial_t \hat w,\,\partial_\theta \hat w\in L^2((T_0,\infty)\times I),
		\qquad
		e^{-t}\hat w\in L^2((T_0,\infty)\times I).
		\]
		In particular,
		\[
		e^{-2t}\hat\rho\in L^2((T_0,\infty)\times(0,\theta_0)),
		\]
		so the right-hand sides in the strip equations are square-integrable and
		exponentially decaying.
		
		Next define the global field on the full strip by
		\[
		\hat U(t,\theta)=
		\begin{cases}
			\hat u(t,\theta), & 0<\theta<\theta_0,\\[2mm]
			\hat v(t,\theta), & \theta_0<\theta<2\pi.
		\end{cases}
		\]
		The transmission conditions \(u=v\) and \(\partial_\theta u=\partial_\theta v\) on
		\(\theta=\theta_0\), together with periodic matching at \(0\sim2\pi\), imply
		that for a.e. \(t\), one has
		\[
		\hat U(t,\cdot)\in H^1_{\mathrm{per}}(0,2\pi).
		\]
		Moreover,
		\[
		(\partial_t^2+\partial_\theta^2)\hat U=\hat g(t,\theta),
		\]
		where
		\[
		\hat g(t,\theta)=
		\begin{cases}
			-\dfrac{e^{-2t}}{h}\hat\rho(t,\theta), & 0<\theta<\theta_0,\\[2mm]
			0, & \theta_0<\theta<2\pi,
		\end{cases}
		\]
		and \(\hat g\in L^2((T_0,\infty)\times(0,2\pi))\).
		
		Expanding in Fourier series in \(\theta\),
		\[
		\hat U(t,\theta)=\sum_{n\in\mathbb Z} a_n(t)e^{in\theta},
		\qquad
		\hat g(t,\theta)=\sum_{n\in\mathbb Z} g_n(t)e^{in\theta},
		\]
		each Fourier coefficient satisfies
		\[
		a_n''(t)-n^2a_n(t)=g_n(t).
		\]
		For \(n\neq0\), the homogeneous solutions are \(e^{\pm |n|t}\). Since
		\(\hat U_t\in L^2((T_0,\infty)\times(0,2\pi))\), the growing mode
		\(e^{|n|t}\) is excluded. Because \(g_n\in L^2\) and decays exponentially, the
		variation-of-constants formula yields
		\[
		a_n(t)=O(e^{-|n|t}),
		\qquad n\neq0.
		\]
		
		For \(n=0\), we have
		\[
		a_0''(t)=g_0(t),
		\]
		with \(g_0\in L^1(T_0,\infty)\cap L^2(T_0,\infty)\). Hence \(a_0'(t)\) has a
		finite limit as \(t\to\infty\). Since \(a_0'\in L^2(T_0,\infty)\), that limit
		must be zero, and therefore \(a_0(t)\to C_u\) for some constant \(C_u\). It
		follows that
		\[
		\hat U(t,\theta)=C_u+O(e^{-t})
		\qquad\text{as }t\to\infty,
		\]
		and therefore
		\[
		\hat u(t,\theta)\to C_u
		\quad (0<\theta<\theta_0),
		\qquad
		\hat v(t,\theta)\to C_u
		\quad (\theta_0<\theta<2\pi).
		\]
		
		It remains to treat \(\hat\rho\), which satisfies
		\[
		(\partial_t^2+\partial_\theta^2)\hat\rho
		=
		\lambda e^{-2t}\hat\rho
		\qquad\text{in }(T_0,\infty)\times(0,\theta_0).
		\]
		Using the cosine basis on \((0,\theta_0)\),
		\[
		\hat\rho(t,\theta)
		=
		b_0(t)+\sum_{m\ge1}
		b_m(t)\cos\!\Big(\frac{m\pi}{\theta_0}\theta\Big),
		\]
		projection onto each mode gives
		\[
		b_m''(t)-\mu_m b_m(t)=f_m(t),
		\qquad
		\mu_m=\Big(\frac{m\pi}{\theta_0}\Big)^2,
		\]
		where \(f_m(t)\) is exponentially decaying and belongs to \(L^2(T_0,\infty)\).
		As before, the growing homogeneous solutions are excluded by the
		\(H^1\)-boundedness of \(\rho\), and therefore
		\[
		b_m(t)=O(e^{-\sqrt{\mu_m}\,t}),
		\qquad m\ge1.
		\]
		
		For the zero mode,
		\[
		b_0''(t)=f_0(t),
		\]
		with \(f_0\in L^1(T_0,\infty)\cap L^2(T_0,\infty)\). Thus \(b_0'(t)\to0\), and
		hence \(b_0(t)\to C_\rho\) for some constant \(C_\rho\). Therefore
		\[
		\hat\rho(t,\theta)=C_\rho+O(e^{-\pi t/\theta_0})
		\qquad\text{as }t\to\infty.
		\]
		
		Returning to the original variable \(r=e^{-t}\), we conclude that
		\[
		u(r,\theta)\to C_u,
		\qquad
		v(r,\theta)\to C_u,
		\qquad
		\rho(r,\theta)\to C_\rho,
		\qquad\text{as }r\to0.
		\]
		This proves the result.
	\end{proof}
	
	\section{Asymptotic evaluation of layer potentials}	
	\begin{lemma}\label{lem: asymptotics S_D^k and K_D^k}
		Assume that $\partial D$ is smooth (at least $C^2$), let $z\in\mathbb{C}$ satisfy $\Im z>0$, let $h\in\mathbb{R}^+$, and let $\psi\in C^1(\partial D)$. Then, as $h\to0$,
		\[
		\SS^{\frac{z}{h}}[\psi](x)
		=
		-\frac{i h}{2z}\,\psi(x)+O(h^2),
		\qquad x\in\partial D,
		\]
		and
		\[
		\KK^{*,\frac{z}{h}}[\psi](x)=O(h),
		\qquad x\in\partial D.
		\]
	\end{lemma}
	
	\begin{proof}
		Set
		\[
		k=\frac{z}{h},
		\qquad \Im z>0,
		\]
		and fix \(x\in \partial D\). Since
		\[
		\Gamma^{\frac{z}{h}}(x-y)
		=
		-\frac{e^{iz|x-y|/h}}{4\pi |x-y|},
		\]
		we have
		\[
		\SS^{\frac{z}{h}}[\psi](x)
		=
		-\int_{\partial D}
		\frac{e^{iz|x-y|/h}}{4\pi |x-y|}
		\psi(y)\,d\sigma(y).
		\]
		
		We first derive the asymptotics for \(\SS^{\frac{z}{h}}\). Choose \(\delta>0\) small enough so that, in a neighbourhood of \(x\), the boundary \(\partial D\) may be represented as a smooth graph over the tangent plane at \(x\). Splitting the integral into local and nonlocal parts, we write
		\[
		\SS^{\frac{z}{h}}[\psi](x)
		=
		-\int_{\partial D\cap B_\delta(x)}
		\frac{e^{iz|x-y|/h}}{4\pi |x-y|}
		\psi(y)\,d\sigma(y)
		-
		\int_{\partial D\setminus B_\delta(x)}
		\frac{e^{iz|x-y|/h}}{4\pi |x-y|}
		\psi(y)\,d\sigma(y),
		\]
		where \(B_\delta(x)\) denotes the Euclidean ball of radius \(\delta\) centred at \(x\).
		
		On \(\partial D\setminus B_\delta(x)\), one has \(|x-y|\ge \delta\), hence
		\[
		\left|
		\int_{\partial D\setminus B_\delta(x)}
		\frac{e^{iz|x-y|/h}}{4\pi |x-y|}
		\psi(y)\,d\sigma(y)
		\right|
		\le
		C e^{-\Im z\,\delta/h}\|\psi\|_{L^\infty(\partial D)}
		=
		O(e^{-c/h})
		\]
		for some \(c>0\). Thus it remains to analyse the local contribution.
		
		After translation and rotation, we may assume that \(x=0\), that the tangent plane at \(x\) is \(\{x_3=0\}\), and that near \(x\) the boundary is parametrised by
		\[
		y(u)=(u,\eta(u)),
		\qquad u\in U\subset\mathbb{R}^2,
		\]
		where \(\eta(0)=0\), \(\nabla \eta(0)=0\), and \(\eta(u)=O(|u|^2)\). Then
		\[
		|y(u)|=|u|+O(|u|^3),
		\qquad
		d\sigma(u)=(1+O(|u|^2))\,du,
		\]
		and, since \(\psi\in C^1(\partial D)\),
		\[
		\psi(y(u))=\psi(x)+O(|u|).
		\]
		Therefore,
		\begin{align*}
			&\int_{\partial D\cap B_\delta(x)}
			\frac{e^{iz|x-y|/h}}{4\pi |x-y|}
			\psi(y)\,d\sigma(y) \\
			&=
			\int_{|u|<\delta'}
			\frac{e^{iz|y(u)|/h}}{4\pi |y(u)|}
			\psi(y(u))(1+O(|u|^2))\,du,
		\end{align*}
		for some \(\delta'>0\). Using the expansions above, we obtain
		\[
		\frac{e^{iz|y(u)|/h}}{|y(u)|}\psi(y(u))(1+O(|u|^2))
		=
		\frac{e^{iz|u|/h}}{|u|}\psi(x)+R_h(u),
		\]
		where \(R_h(u)\) satisfies
		\[
		|R_h(u)|
		\le
		C e^{-\Im z\,|u|/h}.
		\]
		Indeed, the errors coming from
		\[
		\psi(y(u))-\psi(x)=O(|u|),\qquad
		\frac{1}{|y(u)|}-\frac{1}{|u|}=O(|u|),\qquad
		d\sigma(u)-du=O(|u|^2)\,du,
		\]
		all yield terms bounded by \(C e^{-\Im z |u|/h}\), and hence contribute \(O(h^2)\) after integration over \(|u|<\delta'\). Thus
		\[
		\SS^{\frac{z}{h}}[\psi](x)
		=
		-\psi(x)\int_{\mathbb{R}^2}\frac{e^{iz|u|/h}}{4\pi |u|}\,du
		+
		O(h^2),
		\]
		where the replacement of the local patch by \(\mathbb{R}^2\) introduces only an exponentially small error.
		
		It remains to compute the leading integral. Passing to polar coordinates,
		\begin{align*}
			\int_{\mathbb{R}^2}\frac{e^{iz|u|/h}}{4\pi |u|}\,du
			&=
			\frac{1}{4\pi}\int_0^{2\pi}\int_0^\infty e^{izr/h}\,dr\,d\theta \\
			&=
			\frac12\int_0^\infty e^{izr/h}\,dr.
		\end{align*}
		Since \(\Im z>0\),
		\[
		\int_0^\infty e^{izr/h}\,dr
		=
		\frac{i h}{z},
		\]
		and therefore
		\[
		\int_{\mathbb{R}^2}\frac{e^{iz|u|/h}}{4\pi |u|}\,du
		=
		\frac{i h}{2z}.
		\]
		We conclude that
		\[
		\SS^{\frac{z}{h}}[\psi](x)
		=
		-\frac{i h}{2z}\,\psi(x)+O(h^2).
		\]
		
		We now turn to \(\KK^{*,\frac{z}{h}}\). By definition,
		\[
		\KK^{*,\frac{z}{h}}[\psi](x)
		=
		\int_{\partial D}
		\frac{\partial \Gamma^{\frac{z}{h}}(x-y)}{\partial \nu(x)}
		\psi(y)\,d\sigma(y).
		\]
		As above, we split the integral into local and nonlocal parts:
		\[
		\KK^{*,\frac{z}{h}}[\psi](x)
		=
		\int_{\partial D\cap B_\delta(x)}
		\frac{\partial \Gamma^{\frac{z}{h}}(x-y)}{\partial \nu(x)}
		\psi(y)\,d\sigma(y)
		+
		\int_{\partial D\setminus B_\delta(x)}
		\frac{\partial \Gamma^{\frac{z}{h}}(x-y)}{\partial \nu(x)}
		\psi(y)\,d\sigma(y).
		\]
		On \(\partial D\setminus B_\delta(x)\), the same exponential decay gives
		\[
		\int_{\partial D\setminus B_\delta(x)}
		\frac{\partial \Gamma^{\frac{z}{h}}(x-y)}{\partial \nu(x)}
		\psi(y)\,d\sigma(y)
		=
		O(e^{-c/h}).
		\]
		
		For the local part, we compute
		\[
		\frac{\partial \Gamma^{\frac{z}{h}}(x-y)}{\partial \nu(x)}
		=
		\frac{e^{iz|x-y|/h}}{4\pi}
		\left(
		\frac{1}{|x-y|^2}
		-
		\frac{iz}{h\,|x-y|}
		\right)
		\frac{(x-y)\cdot \nu(x)}{|x-y|}.
		\]
		Since \(\partial D\) is smooth,
		\[
		(x-y)\cdot \nu(x)=O(|x-y|^2)
		\qquad\text{as }y\to x.
		\]
		Hence, for \(y\in \partial D\cap B_\delta(x)\),
		\[
		\left|
		\frac{\partial \Gamma^{\frac{z}{h}}(x-y)}{\partial \nu(x)}
		\right|
		\le
		C e^{-\Im z\,|x-y|/h}
		\left(
		\frac{1}{|x-y|}+\frac{|z|}{h}
		\right).
		\]
		
		Using local surface polar coordinates near \(x\), the surface measure satisfies
		\[
		d\sigma(y)\sim r\,dr\,d\theta,
		\qquad r=|x-y|.
		\]
		Therefore,
		\begin{align*}
			\left|
			\int_{\partial D\cap B_\delta(x)}
			\frac{\partial \Gamma^{\frac{z}{h}}(x-y)}{\partial \nu(x)}
			\psi(y)\,d\sigma(y)
			\right|
			&\le
			C\|\psi\|_{L^\infty(\partial D)}
			\int_0^\delta e^{-\Im z\,r/h}
			\left(
			\frac{1}{r}+\frac{|z|}{h}
			\right)r\,dr \\
			&=
			C\|\psi\|_{L^\infty(\partial D)}
			\int_0^\delta e^{-\Im z\,r/h}
			\left(
			1+\frac{|z|}{h}r
			\right)\,dr.
		\end{align*}
		Now set \(r=hs\). Then
		\[
		\int_0^\delta e^{-\Im z\,r/h}
		\left(
		1+\frac{|z|}{h}r
		\right)\,dr
		=
		h\int_0^{\delta/h} e^{-\Im z\,s}(1+|z|s)\,ds.
		\]
		Since \(\Im z>0\), the latter integral is bounded uniformly in \(h\), and we infer that
		\[
		\int_0^\delta e^{-\Im z\,r/h}
		\left(
		1+\frac{|z|}{h}r
		\right)\,dr
		=
		O(h).
		\]
		Combining the local and nonlocal estimates, we conclude that
		\[
		\KK^{*,\frac{z}{h}}[\psi](x)=O(h).
		\]
		This proves the lemma.
	\end{proof}

	\section{Computations for a sphere}\label{sec: Appendix sphere}
	\paragraph{Diagonalisation of $\Lambda_h(z)$}
	Let $D\subset\mathbb{R}^3$ be the unit ball and let $\{Y_{\ell m}\}_{\ell\ge 0,\,-\ell\le m\le \ell}$ denote the spherical harmonics, which form an orthonormal basis of $L^2(\partial D)$. Since the layer potentials are rotationally invariant for the sphere, they diagonalise in this basis. Together with the classical spherical harmonic expansion of the Green function, this yields
	\begin{align*}
		\SS\!\left[Y_{\ell m}\right] &= -\frac{1}{2\ell+1}\,Y_{\ell m}, \\
		\KK^*\!\left[Y_{\ell m}\right] &= \frac{1}{2(2\ell+1)}\,Y_{\ell m}, \\
		\SS^{k}\!\left[Y_{\ell m}\right]
		&=  -i k\,h_\ell^{(1)}(k)\, j_\ell(k)\,Y_{\ell m},\\
		\left(-\tfrac12 I+\KK^{*,k}\right)\!\left[Y_{\ell m}\right]
		&=  -i k^2\,h_\ell^{(1)}(k)\,j_\ell'(k)\,Y_{\ell m},
	\end{align*}
	where $j_\ell$ and $h_\ell^{(1)}$ denote the spherical Bessel and spherical Hankel functions, respectively.
	
	Set $k:=z/h$. Using the above identities, the operator
	\[
	\Lambda_h(z)
	=
	\left( z^2 I+\tfrac12 I +\KK^*\right)
	\left( -\tfrac12 I +\KK^{*,k} \right)
	-
	\left( -\tfrac14 I +(\KK^*)^2 \right)
	\SS^{-1}\SS^{k}
	\]
	(and therefore also the operator $\tilde{\Lambda}_h(z)$) diagonalises in the spherical harmonics basis:
	\[
	\Lambda_h(z)\left[Y_{\ell m}\right]
	=
	\lambda_\ell^{(h)}(z)\,Y_{\ell m},
	\]
	where
	\[
	\lambda_\ell^{(h)}(z)
	:=
	i\,k\,h_\ell^{(1)}(k)
	\left[
	-k\!\left(z^2+\frac{\ell+1}{2\ell+1}\right) j_\ell'(k)
	+
	\frac{\ell(\ell+1)}{2\ell+1}\, j_\ell(k)
	\right].
	\]
	
	\medskip
	
	\paragraph{Modal expansion}
	Assume that $\lambda_\ell^{(h)}(z)\neq 0$ for all $\ell\ge1$. Then
	\[
	\Lambda_h(z)^{-1}
	=
	\sum_{\ell=0}^{\infty}
	\frac{1}{\lambda_\ell^{(h)}(z)}\,P_\ell,
	\]
	where $P_\ell$ denotes the orthogonal projector onto $\operatorname{span}\{Y_{\ell m}:\,|m|\le \ell\}$. For a given function $f\in H^{-\frac12}(\partial D)$ we have
	\[
	P_\ell[f]
	=
	\sum_{m=-\ell}^{\ell}
	\langle f,Y_{\ell m}\rangle_{L^2(\partial D)}\,Y_{\ell m}.
	\]
	Hence
	\begin{align*}
		\psi(z)
		&=
		\Lambda_h(z)^{-1}\!\left[\frac{\partial v^{\text{ext}}}{\partial \nu}\right]
		\nonumber\\
		&=
		-\sum_{\ell=0}^{\infty}
		\frac{z^2}{\lambda_\ell^{(h)}(z)}\,
		P_\ell\!\left[\frac{\partial v^{\text{ext}}}{\partial \nu}\right].
	\end{align*}
	
	Using this expansion, we can compute the scattered field. From \eqref{eq: q(psi)} we obtain
	\[
	\left( \frac12 I +\KK^* \right) q
	=
	-\left(-\tfrac12 I+\KK^{*,k}\right)[\psi]
	-
	\frac{\partial v^{\text{ext}}}{\partial \nu}.
	\]
	Projecting onto each spherical harmonic subspace gives
	\begin{align*}
		\left( \frac12 I +\KK^* \right) P_\ell[q]
		&=
		\frac{z^2}{\lambda_\ell^{(h)}(z)}
		\left(-\tfrac12 I+\KK^{*,k}\right)
		P_\ell\!\left[\frac{\partial v^{\text{ext}}}{\partial \nu}\right]
		-
		P_\ell\!\left[\frac{\partial v^{\text{ext}}}{\partial \nu}\right].
	\end{align*}
	Therefore
	\[
	P_\ell[q]
	=
	(2\ell+1)
	\left(
	\frac{k j_\ell'(k)-\ell j_\ell(k)}
	{-k\!\left(z^2(2\ell+1)+\ell+1\right) j_\ell'(k)+\ell(\ell+1) j_\ell(k)}
	\right)
	P_\ell\!\left[\frac{\partial v^{\text{ext}}}{\partial \nu}\right].
	\]
	
	Consequently, the scattered field admits the classical spherical harmonic expansion
	\begin{align*}\label{eq: scattered field sphere}
		v(x)-v^{\text{ext}}(x)
		&=
		\SS[q](x)
		\nonumber\\
		&=
		\sum_{\ell=0}^{\infty}
		\frac{k j_\ell'(k)-\ell j_\ell(k)}
		{-k\!\left(z^2(2\ell+1)+\ell+1\right) j_\ell'(k)+\ell(\ell+1) j_\ell(k)}
		\frac{1}{|x|^{\ell+1}}
		P_\ell\!\left[\frac{\partial v^{\text{ext}}}{\partial \nu}\right]
		\!\left(\frac{x}{|x|}\right).
	\end{align*}
	This representation provides an explicit modal expansion for the scattered field and can be used to compute optical observables for a spherical nanoparticle.
	
	\begin{remark}
		If $z_j$ satisfies $\lambda_\ell^{(h)}(z_j)=0$, then
		\[
		\operatorname{Res}_{z=z_j}\tilde{\Lambda}_h(z)^{-1}
		=
		\frac{1}{{\lambda'_\ell}^{(h)}(z_j)}\,P_\ell.
		\]
		Since $h_\ell^{(1)}(k)$ has no zeros and the spherical Bessel functions $j_\ell$ and $j_\ell'$ do not vanish simultaneously for $k\neq0$, the zeros of $\lambda_\ell^{(h)}(z)$ are simple. Consequently, $\tilde{\Lambda}_h(z)^{-1}$ possesses only simple poles.
	\end{remark}
	
	\paragraph{Nonlocal plasmonic modes and eigenvalues}
	The equation
	\[
	\lambda_\ell^{(h)}(z)=0
	\]
	defines the dispersion relation for the $\ell$-th nonlocal plasmonic mode of the sphere. The mode $\ell=0$ corresponds to a constant boundary density and does not generate plasmonic resonances. Each integer $\ell\ge1$ therefore corresponds to a family of resonant modes whose frequencies are determined by the roots of this equation.
	
	These modes can be obtained directly by separation of variables. Writing the solutions in spherical coordinates $(r,\theta)$, the charge density $\rho$, the interior potential $u$, and the exterior potential $v$ takes the form
	\[
	\rho(r,\theta)=R_\ell(r)\,Y_{\ell m}(\theta),\qquad
	u(r,\theta)=U_\ell(r)\,Y_{\ell m}(\theta),\qquad
	v(r,\theta)=V_\ell(r)\,Y_{\ell m}(\theta),
	\]
	where $Y_{\ell m}$ denote the spherical harmonics.
	
	Substituting this ansatz into the governing system shows that the radial components satisfy ordinary differential equations. In the interior of the particle, the charge density satisfies the Helmholtz equation
	\[
	\Delta \rho + k^2 \rho =0,
	\qquad k=\frac{z}{h},
	\]
	and regularity at the origin gives
	\[
	R_\ell(r)=j_\ell(kr),
	\]
	where $j_\ell$ denotes the spherical Bessel function. The interior potential satisfies
	\[
	h\Delta u=-\rho,
	\]
	and therefore takes the form
	\[
	U_\ell(r)=A_\ell r^\ell + B_\ell j_\ell(kr).
	\]
	In the exterior region the potential is harmonic and decays at infinity, giving
	\[
	V_\ell(r)=C_\ell r^{-(\ell+1)}.
	\]
	
	The boundary conditions couple these radial solutions and can be used to determine the constants $A_\ell$, $B_\ell$, and $C_\ell$.
\end{document}